\documentclass[reqno,11pt]{article}
\usepackage{amsmath,amssymb}
\usepackage{mathrsfs,enumerate}
\usepackage[
color,
final
]        
{showkeys}

\numberwithin{equation}{section}  
\newtheorem{theorem}{Theorem}[section]  
\newtheorem{lemma}[theorem]{Lemma}

\newtheorem{proposition}[theorem]{Proposition}

\newtheorem{remark}[theorem]{Remark}

\newtheorem{definition}[theorem]{Definition}

\newtheorem{assumption}[theorem]{Assumption}
\newenvironment{proof}{\removelastskip\par\medskip   
\noindent{\em Proof.} \rm}{\penalty-20\null\hfill$\square$\par\medbreak}
\newcommand{\R}{\mathbb{R}}
\newcommand{\Q}{\mathbb{Q}}
\newcommand{\N}{\mathbb{N}}

\newcommand{\E}{\mathbb{E}}
\newcommand{\bbP}{\mathbb{P}}

\newcommand{\cE}{{\ensuremath{\mathcal E}} }

\newcommand{\cN}{{\ensuremath{\mathcal N}} }

\newcommand{\BB}{\mathscr{B}}

\newcommand{\ee}{{\mbox{\boldmath$e$}}}

\newcommand{\ii}{{\mbox{\boldmath$i$}}}

\newcommand{\rr}{{\mbox{\boldmath$r$}}}

\renewcommand{\ss}{{\mbox{\boldmath$s$}}}

\renewcommand{\tt}{{\mbox{\boldmath$t$}}}

\newcommand{\vv}{{\mbox{\boldmath$v$}}}

\newcommand{\tauV}{{\kern-3pt\tau}}

\newcommand{\Haus}[1]{{\mathscr H}^{#1}}     
\newcommand{\Leb}[1]{{\mathscr L}^{#1}}      
\newcommand{\la}{{\langle}}                  
\newcommand{\ra}{{\rangle}}
\newcommand{\eps}{\varepsilon}

\newcommand{\Semi}[2]{{{\mathscr S}{#1}({#2})}}
\newcommand{\Semii}[3]{{{\mathscr S}_{(#3)}{#1}({#2})}}
\newcommand{\Semit}[3]{{{\mathscr S}_{#3}{#1}({#2})}}
\newcommand{\Semitn}[3]{{{\mathscr S}^n_{#3}{#1}({#2})}}

\newcommand{\law}[2]{\nu_{#1}^{#2}}

\newcommand{\Expectation}[1]{\E\big[#1\big]}

\newcommand{\restr}[1]{\lower3pt\hbox{$|_{#1}$}}

\newcommand{\BorelSets}[1]{\BB(#1)}
\newcommand{\GBorelSets}[1]{\BB_G(#1)}
\newcommand{\Probabilities}[1]{\mathscr P(#1)}          
\newcommand{\ProbabilitiesTwo}[1]{\mathscr P_2(#1)}     

\newcommand{\Pc}[2]{\overline{#1}\kern-2pt^{\vphantom 0}_{#2}}
\newcommand{\Pcb}[2]{\underline{\phantom{..}}\kern-6pt #1_{#2}}

\newcommand{\RelativeEntropy}[2]{\mathcal H(#1|#2)}

\newcommand{\nchi}{{\raise.3ex\hbox{$\chi$}}}
\newcommand{\see}{{\mbox{\scriptsize\boldmath$e$}}}

\newcommand{\Hn}{H_n}

\title{Existence and Stability for Fokker-Planck equations
with log-concave reference measure}
\author{Luigi Ambrosio\thanks{\textsf{l.ambrosio@sns.it}}\\
  Scuola Normale Superiore,
  Pisa
  \and
  Giuseppe Savar\'e\thanks{\textsf{giuseppe.savare@unipv.it}}\\
  Dipartimento di Matematica,
  Universit\`a di Pavia
  \and
  Lorenzo Zambotti\thanks{\textsf{zambotti@ccr.jussieu.fr}}\\
  LPMA, Universit\'e Paris VI}
\date{}

\begin{document}

\maketitle

\begin{abstract}
We study Markov processes associated with stochastic differential
equations, whose non-linearities are gradients of convex
functionals. We prove a general result of existence of such Markov
processes and a priori estimates on the transition probabilities.
The main result is the following stability property: if the
associated invariant measures converge weakly, then the Markov
processes converge in law. The proofs are based on the
interpretation of a Fokker-Planck equation as the steepest descent
flow of the relative Entropy functional in the space of probability
measures, endowed with the Wasserstein distance. Applications
include stochastic partial differential equations and convergence of
equilibrium fluctuations for a class of random interfaces.


\end{abstract}


\section{Introduction and main results}

In the seminal paper \cite{jko}, Jordan-Kinderlehrer-Otto have given
a remarkable interpretation of the solution to a linear
Fokker-Planck equation as the steepest descent flow of the relative
Entropy functional in the space of probability measures, endowed
with the Wasserstein distance. The book \cite{ags} by
Ambrosio-Gigli-Savar\'e has provided a general theory of gradient
flows in the Wasserstein space of probability measures, including
linear and non-linear PDE's, in finite and infinite dimension.

In this paper we want to investigate the probabilistic counterpart
of such results. The approach is analytical and based on techniques
from Calculus of Variations and Optimal Transport Problems; however
several results have important consequences on existence and in
particular convergence of Markov processes being reversible with
respect to a \emph{log-concave} probability measure.

Following \cite{ags}, we interpret the solution $(\mu_t)_{t\geq 0}$
of a Fokker-Planck equation with convex potential, as a curve in the
space of probability measures, solving a suitable \emph{differential
variational inequality}. We obtain interesting estimates on $\mu_t$
which have, to our knowledge, no direct probabilistic proof, and are
very useful in the study of the time-homogeneous Markov process
$(X_t)_{t\geq 0}$ whose one-time distributions are $(\mu_t)_{t\geq
0}$.

\subsection{The main results}
We consider a separable Hilbert space $H$, which could be finite or
infinite dimensional, whose scalar product and norm will be
respectively denoted by $\la\cdot,\cdot\ra$ and $\|\cdot\|$. We
denote by $\Probabilities{H}$ the set of all probability measures on
$H$, endowed with the Borel $\sigma$-algebra.

We consider a probability measure $\gamma$ on $H$ with the
following property:
\begin{assumption}\label{logc}
$\gamma$ is {\it log-concave}, i.e.
for all pairs of open sets $B,\, C\subset H$
\begin{equation}\label{deflogconc}
\log\gamma\left((1-t)B+tC\right)\geq (1-t)\log\gamma(B)+t\log\gamma(C)
\qquad\forall t\in (0,1).
\end{equation}
\end{assumption}
The class of log-concave probability measures includes all
measures of the form (here $\Leb{k}$ stands for Lebesgue measure)
\begin{equation}
  \label{eq:basic_example}
  \gamma :=  \frac 1Z \, e^{-V} \Leb{k},
  \qquad\text{where $V:H=\R^k\to\R$ is convex and $Z:= \int_{\R^k} e^{-V} \, dx<+\infty$},
\end{equation}
all Gaussian measures, all Gibbs measures on a finite lattice with
convex Hamiltonian; see Proposition~\ref{charalog} and the Appendix
for more information on the class of log-concave probability
measures.

We denote the support of $\gamma$ by $K=K(\gamma)$ and
the smallest closed affine subspace of $H$ containing $K$
by $A=A(\gamma)$. We write canonically
\begin{equation}\label{defah}
A \, = \, H^0 \, + \, h^0, \qquad
h^0\in K, \quad \|h^0\| \, \leq \, \|k\| \quad \forall
\ k\in K,
\end{equation}
so that $h^0=h^0(\gamma)$ is the element of minimal norm
in $K$ and $H^0=H^0(\gamma)$ is a closed linear subspace of $H$.
As in the Gaussian case, we will say that $\gamma$ in \emph{non-degenerate}
if $H^0(\gamma)=H$.

We want to consider a stochastic processes with values in $A(\gamma)$
and reversible with respect to $\gamma$. We now state a first result
which determines such process in a canonical way.
We denote by $C_b(H)$ the space of bounded continuous functions in
$H$ and by $C_b^1(A(\gamma))$ the space of all
$\Phi:A(\gamma)\mapsto\R$ which are bounded, continuous and
Fr\'echet differentiable with bounded continuous gradient
$\nabla\Phi:A(\gamma)\mapsto H^0(\gamma)$ (notice
that all functions in $C_b^1(A(\gamma))$ are Lipschitz continuous).

We set $\Omega:=C([0,+\infty[;K)\subset K^{[0,+\infty[}$, and we denote by
$X_t:K^{[0,+\infty[}\to K$ the coordinate process $X_t(\omega):=\omega_t$, $t\geq 0$.
We shall endow $\Omega$ with the Polish topology of uniform convergence on bounded
subsets of $[0,+\infty[$, and the relative Borel $\sigma$-algebra. On $K^{[0,+\infty[}$
we shall consider the canonical $\sigma$-algebra generated by cylindrical sets and, for
probability measures in $K^{[0,+\infty[}$, the convergence induced by the duality
with continuous cylindrical functions of the form $f(X_{t_1},\ldots,X_{t_n})$,
with $f\in C_b(K^n)$.
\begin{theorem}[Markov process and Dirichlet form associated to $\gamma$]\label{main1}
Let $\gamma$ be a log-concave probability measure on $H$ and let $K$
be its support. Then:
\begin{itemize}
\item[(a)]
The bilinear form ${\cal E}={\cal E}_{\gamma,\|\cdot\|}$ given by
\begin{equation}\label{diri}
{\cal E}(u,v) \, := \, \int_{K} \la\nabla u,\nabla
v\ra_{H^0(\gamma)} \, d\gamma, \qquad u,\,v\in C^1_b(A(\gamma)),
\end{equation}
is closable in $L^2(\gamma)$ and its closure $(\cE,D(\cE))$ is a
symmetric Dirichlet Form. Furthermore, the associated semigroup
$(P_t)_{t\geq 0}$ in $L^2(\gamma)$ maps $L^\infty(\gamma)$ in
$C_b(K)$.
\item[(b)] There exists a unique Markov family $(\bbP_x:x\in K)$ of probability measures
on $K^{[0,+\infty[}$ associated with $\cE$. More precisely, $\E_x[f(X_t)]=P_tf(x)$
for all bounded Borel functions and all $x\in K$. Moreover, $x\mapsto\bbP_x$ is continuous.
\item[(c)] For all $x\in K$, $\bbP_x^*\left(C(]0,+\infty[;H)\right)=1$ and
$\E_x[\|X_t-x\|^2]\to 0$ as $t\downarrow 0$. Moreover, $\bbP_x^*\left(C([0,+\infty[;H)\right)=1$
for $\gamma$-a.e. $x\in K$.
\item[(d)] $(\bbP_x:x\in K)$ is reversible with respect to $\gamma$,
i.e. the transition semigroup $(P_t)_{t\geq 0}$ is symmetric in
$L^2(\gamma)$; moreover $\gamma$ is invariant for $(P_t)$, i.e.
$\gamma(P_tf)=\gamma(f)$ for all $f\in C_b(K)$ and $t\geq 0$.
\end{itemize}
\end{theorem}
An example in $H:=\R^k$ of the above setting is provided by
\eqref{eq:basic_example} when the potential $V:\R^k\to\R$ is convex
with Lipschitz continuous gradient $\nabla V:\R^k\to\R^k$. Then
$\gamma$ is log-concave, see Proposition \ref{charalog}, and the
process $X$ is a solution of the Stochastic Differential Equation
(SDE):
\begin{equation}\label{gs1}
dX_t \, = \, -\nabla V(X_t) \, dt + \, \sqrt 2 \, dW, \qquad
X_0(x)=x,
\end{equation}
where $W$ is a $\R^k$-valued Brownian motion.
One can also consider a convex $V\in C^{1,1}(U)$, where
$U\subset\R^k$ is a convex open set, and $V\equiv+\infty$
on $\R^k\setminus U$. Then $X$ solves the SDE
with reflection at the boundary $\partial U$ of $U$:
\begin{equation}\label{gs2}
dX_t \, = \, -\nabla V(X_t) \, dt + \, \sqrt 2 \, dW
+ {\bf n}(X_t) \, dL_t, \qquad
X_0(x)=x,
\end{equation}
where ${\bf n}$ is an inner normal vector to $\partial U$ and $L$ is
a continuous monotone non-decreasing process which increases only
when $X_t\in\partial U$. Equations like (\ref{gs1}) and (\ref{gs2})
with convex potentials arise in the theory of random interfaces. The
invariant measure $\gamma$ is typically a Gibbs measure on a
lattice. Interesting infinite-dimensional examples include
Stochastic PDEs with monotone gradient non-linearities or with
reflection. See subsection \ref{motivations} for an overview of the
literature.

\smallskip
Before stating the next theorem, we define the
relative Entropy functional; for all probability
measures $\mu$ on $H$ we set:
\begin{equation}\label{relent}
\RelativeEntropy{\mu}{\gamma} \, := \, \int_H \rho \, \log\rho\,d\gamma
\end{equation}
if $\mu=\rho\, \gamma$ for some $\rho\in L^1(\gamma)$, and $+\infty$ otherwise.
We recall that $\RelativeEntropy{\cdot}{\gamma}\geq 0$ by Jensen's inequality.

We also define the Wasserstein distance: given two probability measures
 $\mu,\,\nu$ on $H$, we set
\begin{equation}\label{defwas}
W_2(\mu,\nu):=\inf\left\{
\left[\int_{H\times H}\|y-x\|^2\,d\Sigma\right]^{\frac12}:\ \Sigma\in\Gamma(\mu,\nu)
\right\}.
\end{equation}
Here $\Gamma(\mu,\nu)$ is the set of all \emph{couplings} between $\mu$ and $\nu$:
it consists of all probability measures $\Sigma$ on $H\times H$ whose first
and second marginals are respectively $\mu$ and $\nu$, i.e. $\Sigma(B\times H)=\mu(B)$
and $\Sigma(H\times B)=\nu(B)$ for all $B\in\BorelSets{H}$. We set
$$
\ProbabilitiesTwo{H}:=\left\{\mu\in\Probabilities{H}:\
\int_H\Vert x\Vert^2\,d\mu(x)<\infty\right\}.
$$
It turns out that $W_2(\cdot,\cdot)$ is a distance on $\ProbabilitiesTwo{H}$
and that $(\ProbabilitiesTwo{H},W_2)$ is a complete and
separable metric space, whose convergence implies weak convergence,
see for instance \cite[Proposition~7.1.5]{ags}.
Then, we have the following result:
\begin{theorem}[Estimates on transition probabilities]\label{main2}
Let $\gamma$ be a log-concave probability measure on $H$ and let
$(\bbP_x)$ be as in Theorem \ref{main1}. Fix $x\in K$ and denote the
law of $X_t$ under $\bbP_x$ by $\law tx$. Then $[0,+\infty[\times K
\ni (t,x) \mapsto \nu_t^x\in\ProbabilitiesTwo{H}$ is continuous and
\[
\RelativeEntropy{\law tx}{\gamma} \, \leq \,
\inf_{\sigma\in\ProbabilitiesTwo{H}} \left\{
\frac1{2t} \int_H \|y-x\|^2\,d\sigma(y)
\ + \ \RelativeEntropy{\sigma}{\gamma} \right\}
\, < \, +\infty\qquad\forall t>0,
\]
so that $\law tx\ll\gamma$ for all $t>0$, $x\in K$. Moreover,
\[
W_2(\law tx,\law sx) \, \leq \, \sqrt{2\, \RelativeEntropy{\law \eps x}{\gamma}}
\, \sqrt{|t-s|}, \qquad t,\,s\geq\varepsilon,\,\,x\in K.
\]
\end{theorem}
Notice that the estimates given in Theorem~\ref{main2} do not
contain any constant depending on $H$ or on $\gamma$ and appear to
be of a structural nature. In the particular case
$\gamma\in\ProbabilitiesTwo{H}$ we have: $\RelativeEntropy{\law
tx}{\gamma} \, \leq \, \frac1{2t} \, W^2_2(\delta_x,\gamma) \, < \,
+\infty$, $\forall t>0$.

\medskip We consider now a sequence $(\gamma_n)$ of log-concave
probability measures on $H$ such that $\gamma_n$ converge weakly to
$\gamma$. We denote $K_n:=K(\gamma_n)$, $A_n:=A(\gamma_n)$,
$\Hn:=H^0(\gamma_n)$, in the notation of \eqref{defah}. We want to
consider situations where each $\Hn$ is an Hilbert space endowed
with a scalar product $\langle\cdot,\cdot\rangle_{\Hn}$ and an
associated $H$-continuous norm $\|\cdot\|_{\Hn}$ possibly different
from the scalar product and the norm induced by $H$. In order to
ensure that this family of norms converges (in a suitable sense) to
the norm of $H$ as $n\to\infty$, we will make the following
assumption:

\begin{assumption}\label{h_n}
There exists a constant $\kappa\geq 1$ such that
\begin{equation}\label{bastakappa}
\frac{1}{\kappa}\Vert h\Vert_H\leq
\Vert h\Vert_{H_n}\leq\kappa\Vert h\Vert_H
\qquad\forall h\in H_n,\,n\in\N.
\end{equation}
Furthermore, denoting by $\pi_n:H\to H_n$ the orthogonal projections
induced by the scalar product of $H$, we have
\begin{equation}\label{sig}
\lim_{n\to\infty}\Vert\pi_n(h)\Vert_{H_n}=\Vert h\Vert_H
\qquad\forall h\in H.
\end{equation}
\end{assumption}
This assumption guarantees in some weak sense that the geometry
of $H_n$ converges to the geometry of $H$; the case when
all the scalar products coincide with $\langle\cdot,\cdot\rangle_H$,
$H_n\subset H_{n+1}$ and $\cup_n H_n$ is dense in $H$ is obviously
included and will play an important role in the paper.

Let $(\bbP_x^n:x\in K_n)$ (respectively $(\bbP_x:x\in K)$) be the
Markov process in $[0,+\infty[^{K_n}$ associated to $\gamma_n$
(resp. in $[0,+\infty[^K$ associated to $\gamma$) given by
Theorem~\ref{main1}. We denote by $\bbP_{\gamma_n}^n:=\int
\bbP_{x}^n \, d\gamma_n(x)$ (resp. $\bbP_\gamma:= \int \bbP_{x} \,
d\gamma(x)$) the associated stationary measures.

With an abuse of notation, we say that a sequence of measures $({\bf
P}_n)$ on $C([a,b];H)$ converges weakly in $C([a,b];H_w)$ if, for
all $m\in\N$ and $h_1,\ldots,h_m\in H$, the process $(\langle
X_\cdot,h_i\rangle_H, \, i=1,\ldots,m)$ under $({\bf P}_n)$
converges weakly in $C([a,b];\R^m)$ as $n\to\infty$.

In this setting we have the following stability and tightness
result:
\begin{theorem}[Stability and tightness]\label{main3}
Suppose that $\gamma_n\to\gamma$ weakly in $H$ and that the norms of
$\Hn$ satisfy Assumption~\ref{h_n}. Then, for all $x_n\in K_n$ such
that $x_n\to x\in K$ in $H$:
\begin{itemize}
\item[(a)] $\bbP_{x_n}^n\to\bbP_x$ weakly in $H^{[0,+\infty[}$ as $n\to\infty$;
\item[(b)] for all $0<\varepsilon\leq
T<+\infty$, $\bbP_{x_n}^n\to\bbP_x$ weakly in
$C([\varepsilon,T];H_w)$;
\item[(c)] for all $0\leq T<+\infty$, $\bbP_{\gamma_n}^n\to\bbP_{\gamma}$ weakly in
$C([0,T];H_w)$.
\end{itemize}
\end{theorem}
This stability property means that the weak convergence
of the invariant measures $\gamma_n$ and a suitable convergence
of the norms $\|\cdot\|_{\Hn}$ to $\|\cdot\|_H$ imply the convergence
in law of the associated processes, starting from {\it any} initial condition.
Notice also statement (b) makes sense, because Theorem~\ref{main1}(c) gives that our
processes have continuous modifications in $C(]0,+\infty[;H)$ (however, we are
able to prove tightness only for the weak topology of $H$).

\medskip
Finally, our approach yields naturally the following
\begin{theorem}[Uniqueness in $\ProbabilitiesTwo{H}$ of the invariant measure]\label{main4}
Let $\gamma\in\ProbabilitiesTwo{H}$. If $\mu\in\ProbabilitiesTwo{H}$
is an invariant measure of $(P_t)_{t\geq 0}$, i.e.
$\mu(P_tf)=\mu(f)$ for all $f\in C_b(K)$ and $t\geq 0$, then
$\mu=\gamma$.
\end{theorem}

\subsection{Motivations and a survey of the literature}\label{motivations}

Existence and uniqueness for stochastic equations like \eqref{gs1}
and \eqref{gs2} in finite dimension are classical problems in
probability theory, starting from \cite{skorohod} and \cite{tanaka}.
In \cite{cepa}, existence and uniqueness of strong solutions are
proven for a general convex potential $V$. The Dirichlet form
approach is detailed in \cite{fot}.

Natural generalizations of \eqref{gs1} to the infinite dimension are
provided by stochastic partial differential equations (SPDEs): see
Chap. 8 of \cite{dpz2} and \cite{dpr}. SPDEs with reflection, which
generalize \eqref{gs2}, have also been studied: see \cite{nupa},
\cite{za02}, \cite{za03}, \cite{deza}. Unlike the finite-dimensional
case, no general result of existence and uniqueness is known, and in
fact it is not even clear how to define a general notion of
solution.

The main result of this paper is the general stability property of
this class of stochastic processes, given by Theorem \ref{main3}: if
the log-concave invariant measures $\gamma_n$ converge, then the
laws $\bbP^n_x$ of the associated stochastic processes also
converge. In order to appreciate the strength of this result, notice
that convergence of $\gamma_n$ is a much weaker information than
convergence (in any sense) of the drift $\nabla V_n$ in \eqref{gs1}.
In fact, every approach based either on the SDE or on the generator
and the Dirichlet form associated with the process, seems bound to
give only weaker results.

In the stability result, the limit process is identified by the
associated Dirichlet form \eqref{diri}: however, in the general
case, we can not write a stochastic equation for the limit, although
this can be (and has been) done in many interesting situations. Our
approach yields existence of stochastic processes associated with
any Dirichlet form of the gradient type \eqref{diri} with
log-concave reference measure: this also seems to be a new result
(see \cite{ak}).

\medskip Stochastic equations of the form \eqref{gs1} and \eqref{gs2}
are used as models for the random evolution of interfaces; in these
cases the invariant measure is typically a Gibbs measure on a
lattice with convex interaction: see \cite{spohn}, \cite{fusp} and
\cite{fu} for the physical background.

In many interesting cases, the Gibbs measure converges, under a
proper rescaling, to a non-degenerate Gaussian (or related) measure
on some function or distribution space. Convergence in law of the
associated stationary dynamics to the solution of a stochastic
partial differential equation is interpreted as convergence of the
equilibrium fluctuations of the interface around its macroscopic
hydrodynamic limit: see \cite{gos} and \cite{fuol}.

Such convergence results are obtained only in the stationary case
and the proofs use very particular properties of the model. For
instance, the techniques of \cite{fuol} are based on monotonicity
properties and can not be applied to many interesting situations.
Our Theorem~\ref{main3} extends the convergence result to more
general initial conditions and, being based only on the
log-concavity of the invariant measures, can be applied to a large
class of models. For a different (and weaker) approach based on
infinite dimensional integration by parts, see \cite{za04} and
\cite{za}.

Finally, we notice that log-concave measures are still widely used
as models for random interfaces: see \cite{sheffield} and references
therein.

\subsection{Plan of the paper}

We conclude this introduction with a short description of the plan
of the paper: Section~2 is devoted to the introduction of some basic
concepts and terminology, while in Section~3 we illustrate the model
case when $H=\R^k$ and $\nabla V$ is smooth, bounded and Lipschitz:
here almost no technical issue arises and the basic heuristic ideas
can be presented much better. In Section~4 we show the basic
convexity properties of the relative Entropy functional needed to
build in Section~5, by implicit time discretization, a
``Fokker-Planck'' semigroup in the Wasserstein space of probability
measures. Section~6 is devoted to the quite strong stability
properties of this semigroup, and these are used in Section~7 to
establish, starting from the smooth case, the link with Dirichlet
forms. Finally, in Section~8 we canonically build our process in
$K^{[0,+\infty[}$, and deduce its continuity properties from the
continuity properties of its transition probabilities, provided by
the Wasserstein semigroup. Finally, we adapt to our case some
general results from \cite{maro} on the existence of Markov
processes associated to Dirichlet forms to obtain the results stated
in Theorem~\ref{main1}(c).

\section{Notation and preliminary results}
\label{notation}
In this section we fix our main notation and recall the main
results on Wasserstein distance and optimal couplings.

Throughout the paper we consider a real separable Hilbert space $H$. For
$J\subset H$ closed we denote by
${\rm Lip}_b(J)$ the space of all bounded $\varphi:J\mapsto\R$
such that:
\[
[\varphi]_{{\rm Lip(J)}} \, := \, \sup\left\{ \frac{|\varphi(x)-\varphi(y)|}
{\|x-y\|} \ : \ x,\,y\in J,\,\,x\neq y \ \right\} \, < \, +\infty.
\]

\smallskip
\noindent
{\bf Measure-theoretic notation.} If $H$ is a separable Hilbert space,
we shall denote by $\BorelSets{H}$ the Borel $\sigma$-algebra of $H$, and by
$\Probabilities{H}$ the set of (Borel) probability measures in $H$.
Given a Borel map $\rr:H\to H$, the \emph{push forward} $\rr_\#\mu\in\Probabilities{H}$
of $\mu\in\Probabilities{H}$ is defined by $\rr_\#\mu(B):=\mu(\rr^{-1}(B))$ for all
$B\in\BorelSets{H}$.

The set of non-degenerate Gaussian measures on $H$, which all belong
to $\ProbabilitiesTwo{H}$, will be denoted by $G(H)$. Analogously,
we shall denote by $\GBorelSets{H}$ the $\sigma$-ideal of
\emph{Gaussian} null sets, \emph{i.e.} the sets $B\in\BorelSets{H}$
such that $\mu(B)=0$ for all $\mu\in G(H)$. Lebesgue measure in
$\R^k$ will be denoted by $\Leb{k}$.

\smallskip
\noindent
{\bf Wasserstein distance, optimal couplings and maps.}
We have already defined the class of \emph{couplings} between two
probability measures $\mu$ and $\nu$ on $H$ and the Wasserstein distance
$W_2(\mu,\nu)$: see \eqref{defwas}.
Existence of a minimizing $\Sigma$ in \eqref{defwas} is a simple consequence of the tightness
of $\Gamma(\mu,\nu)$; the class of \emph{optimal couplings} will be denoted by
$\Gamma_o(\mu,\nu)$:
\begin{equation}\label{optcou}
\Gamma_o(\mu,\nu) \, := \,\left\{\Sigma\in\Gamma(\mu,\nu):
\ \int_{H\times H}\|y-x\|^2\,d\Sigma = W_2^2(\mu,\nu) \right\}.
\end{equation}

In the special case when $\mu$ vanishes on all Gaussian null sets
(that corresponds to absolute continuity with respect to Lebesgue
measure in finite dimensions) it has been proved in Theorem~6.2.10
of \cite{ags} that there exists a unique optimal coupling $\Sigma$,
and it is induced by an optimal \emph{transport map} $\tt$, namely
$\Sigma=(\ii\times\tt)_\#\mu$ (the proof is based on the fact that
the non-Gateaux differentiability set of a Lipschitz function in $H$
is Gaussian null, see e.g. Theorem 5.11.1 in \cite{bogachev}). We
shall denote this optimal transport map by $\tt_{\mu}^\nu$. This is
one of the infinite-dimensional generalizations (see also
\cite{FeyUst} for another result in Wiener spaces) of the
finite-dimensional result ensuring that whenever
$\mu\in\ProbabilitiesTwo{\R^k}$ is absolutely continuous with
respect to $\Leb{k}$, then there exists a unique optimal transport
map that is also the gradient of a convex function.

When we have a sequence $(\gamma_n)\subset\ProbabilitiesTwo{H}$ as in
Assumption~\ref{h_n}, we can
introduce Wasserstein distances in $\ProbabilitiesTwo{A_n}$ using
two different scalar products: $ \langle \cdot,\cdot\rangle_H$
and $ \langle \cdot,\cdot\rangle_{\Hn }$. The Wasserstein distance
with respect to the former one is
indicated in the standard way $W_2(\cdot,\cdot)$, while we
introduce the notation:
\begin{equation}\label{www}
W_{2,\Hn}^2(\mu,\nu) :=
\inf\left\{\int_{A_n\times A_n}\|y-x\|_{\Hn }^2\,d\Sigma:\
\Sigma\in\Gamma(\mu,\nu)\right\}, \qquad
\mu,\,\nu\in \ProbabilitiesTwo{A_n};
\end{equation}
notice that if $x,\,y\in A_n$ then $x-y\in \Hn$,
so that $\|x-y\|_{\Hn }$ makes sense.
If $\mu,\,\nu$ are supported in $A_n$ we also denote
the class of optimal couplings in $\Gamma(\mu,\nu)$ with respect to the $\Hn $-distance
by $\Gamma_{\Hn,o}(\mu,\nu)$. By \eqref{bastakappa} the two distances
are equivalent.

\smallskip
\noindent {\bf Convergence of measures.} We will use two notions
convergence of measures: first the \emph{weak} convergence in
$\Probabilities{H}$, induced by the duality with $C_b(H)$; second,
the convergence in $\ProbabilitiesTwo{H}$ induced by the Wasserstein
distance. The two definitions are related by the following result
(see \cite{ags}, Theorem~5.1.13 and Remark 7.1.11):
\begin{lemma}\label{l_basic_tight}
If $(\mu_n)\subset\ProbabilitiesTwo{H}$, then
$\mu_n\to\mu$ in $\ProbabilitiesTwo{H}$ if and only if
$\mu_n\to\mu$ weakly and
\begin{equation}\label{cami1}
\lim_{n\to\infty}\int_H\|x\|^2\,d\mu_n=\int_H\|x\|^2\,d\mu.
\end{equation}
\end{lemma}
Notice that, for weakly converging sequences $(\mu_n)$, the convergence of
the second moments \eqref{cami1} is easly seen to be equivalent to
\begin{equation}\label{cami2}
    \lim_{R\uparrow\infty}\limsup_{n\to\infty}
    \int_{\{\|x\|\geq R\}}\|x\|^2\,d\mu_n=0.
\end{equation}
We recall that weak convergence of $\mu_n$ to $\mu$ implies
\begin{equation}\label{fatext}
  \liminf_{n\to\infty}\int_H f\,d\mu_n\geq\int_H f\,d\mu,
\end{equation}
for every lower semicontinuous function $f:H\to (-\infty,+\infty]$
bounded from below. We shall also often use the following extension,
involving integration with respect to a variable function: if $f_n$
are uniformly bounded from below and equi-continuous, we have
\begin{equation}\label{fatext1}
\liminf_{n\to\infty}\int_H f_n\,d\mu_n\geq
\int_H \liminf_{n\to\infty}f_n\,d\mu.
\end{equation}
The proof immediately follows by \eqref{fatext}, with the monotone
approximation with the continuous functions $g_k=\inf\limits_{n\geq k}f_n$.

\medskip\noindent
{\bf Log-concave probability measures and Entropy}.
The concept of log-concavity has been introduced in Assumption \ref{logc}.
Since this concept is crucial in this paper, we recall the following result.
\begin{proposition} [\cite{borell}, \cite{ags}, Theorem~9.4.11]\label{charalog}
Let $H=\R^k$.
Then $\gamma\in\Probabilities{H}$ is log-concave if and only if
it admits the following representation:
\begin{equation}\label{repgamma}
\gamma(B)=\int_{B\cap\{V<+\infty\}}e^{-V}\,d\Haus{d}
\qquad\forall B\in\BorelSets{\R^k},
\end{equation}
where $V:\R^k\to (-\infty,+\infty]$ is a suitable convex and lower
semicontinuous function, $d\geq 0$ is the dimension of $A(\gamma)$,
and $\Haus{d}$ is the $d$-dimensional Hausdorff measure.

If the dimension of $H$ is infinite, then $\gamma\in\Probabilities{H}$ is
log-concave if and only if all the finite dimensional projections of $\gamma$ are log-concave
and therefore admit the representation \eqref{repgamma} for some $V$ and $d$.
\end{proposition}
If $\gamma$ is log-concave, the relative Entropy functional
$\mathcal H(\cdot|\gamma)\, $ \eqref{relent} enjoys a crucial
convexity property in terms of Wasserstein distance, which has been
discovered by McCann in \cite{McCann97} and further extended to the
infinite dimensional case in \cite{ags}.
\begin{proposition}[Displacement convexity of the relative Entropy]
  \label{prop:displ_conv}
  Let $\gamma\in \Probabilities H$ be log-concave
  and let $\mu^0,\,\mu^1\in \ProbabilitiesTwo H$ with finite relative
  entropy.
  Then there exists an optimal coupling $\Sigma\in\Gamma_o(\mu^0,\mu^1)$
  such that the curve in $\ProbabilitiesTwo H$
  \begin{equation}
    \label{eq:Gibbs:20}
    \mu^t:=\big((1-t)\pi^0+t\pi^1\big)_\#\Sigma,\qquad
    \left(\pi^i:(x^0,x^1)\in H\times H\mapsto x^i\in H,\quad i=0,1\right)
  \end{equation}
  satisfies
  \begin{equation}
    \label{eq:Gibbs:19}
    \mathcal H(\mu^t|\gamma)\le (1-t)\mathcal H(\mu^0|\gamma)+
    t\mathcal H(\mu^1|\gamma)\qquad
    \forall\, t\in [0,1].
  \end{equation}
\end{proposition}
When $H=\R^k$ is finite dimensional and $\mu^0$ is absolutely
continuous w.r.t.\ the Lebesgue measure then the optimal coupling
$\Sigma=(\ii\times \tt)_\#\mu^0$ is unique, so that
\begin{equation}
  \label{eq:Gibbs:22}
  \mu^t=\big((1-t)\ii+t\,\tt\big)_\#\mu^0.
\end{equation}

\section{From Fokker-Planck equation to Wasserstein gradient flows}
\label{tuttolega}

In this section we illustrate the known connections between
solutions of the SDE \eqref{gs1}, solutions to Fokker-Planck
equations, Dirichlet semigroups and Wasserstein gradient flows in
the model case when the drift term in the SDE is the bounded
gradient $\nabla V$ of a smooth function $V:\R^k\to\R$ satisfying:
\[
\|\nabla V(x)-\nabla V(y)\| \leq  L \|x-y\| \qquad\forall x,\,y\in\R^k
\]
for some $L>0$. We shall also assume that all derivatives of $V$ are
bounded and that $\gamma=\exp(-V) \, \Leb{k}$ is a log-concave
probability measure in $\R^k$. Notice that this implies that $V$ is
convex, and also (see Appendix \ref{properties}) that there exist
constants $A\in\R$ and $B>0$ such that $V(x)\geq A+B\|x\|$ for all
$x\in\R^k$.

All theories mentioned above have a much larger realm of validity
(for instance, much less regular drift terms in the SDE \eqref{gs1}
are allowed), but for our purposes it suffices to show connections
and a few a priori estimates in the smooth, bounded, Lipschitz case:
more general cases will follows thanks to the stability
Theorem~\ref{main3} (or its Wasserstein counterpart
Theorem~\ref{stabflows}).

\medskip
Let us fix $k$ independent standard Brownian motions
$\{W^1,\ldots,W^k\}$ on a probability space. We consider the
$\R^k$-valued Brownian motion $(W_t)_{t\geq 0}$, where $W=(
W^1,\ldots,W^k)$. Since $\nabla V$ is bounded and Lipschitz
continuous, it is well known that, for all $x\in \R^k$, there exists
a unique solution $(X_t(x): t\geq 0)$ of the SDE
\begin{equation}\label{gs3}
dX_t \, = \, -\nabla V(X_t) \, dt + \, \sqrt 2 \, dW, \qquad
X_0(x)=x.
\end{equation}
Notice that $(X_t(x)-X_t(y), t\geq 0)$ solves almost surely an
ordinary differential equation, since the stochastic terms $dW$
cancel out; then one easily obtains from the convexity of $V$ that
$t\mapsto\|X_t(x)-X_t(y)\|^2$ is non-increasing in $[0,+\infty[$
almost surely. As a consequence, a.s.
\begin{equation}\label{eq:dr1}
\|X_t(x)-X_t(y)\|^2 \leq \|x-y\|^2, \qquad\forall \, x,y\in\R^k, \
t\geq 0.
\end{equation}
For all $x\in\R^k$, $t\geq 0$ and
$\mu_0\in\Probabilities{\R^k}$ we set:
\begin{equation}\label{eq:dr1.5}
\law tx := \text{law of } X_t(x), \qquad \mu_t := \int \law tx \,
d\mu_0(x) \, \in\Probabilities{\R^k}.
\end{equation}
By \eqref{eq:dr1}, the map $x\mapsto\law tx$ is weakly continuous,
and therefore $\mu_t$ is well defined. Moreover, the continuity of
the process $(X_t(x))_{t\geq 0}$ yields weak continuity of
$t\mapsto\law tx$ and $t\mapsto\mu_t$.

It is a trivial consequence of It\^o's formula that $\mu_t$ solves
the Fokker-Planck equation in the sense of distributions in
$]0,+\infty[\times\R^k$:
\begin{equation}\label{fk}
\frac d{dt} \, \mu_t = \Delta\mu_t + \nabla \cdot (\nabla V\mu_t);
\end{equation}
this means that
\begin{equation}\label{fk1}
\frac d{dt} \int_{\R^k} \varphi \ d\mu_t =
\int_{\R^k} \left(\Delta\varphi  -
\langle \nabla V,\nabla\varphi\rangle \right) \, d\mu_t
\qquad\forall\varphi\in C^\infty_c(\R^k),
\end{equation}
and the initial condition at $t=0$ is attained in the following weak sense:
\begin{equation}\label{fk2}
\lim_{t\downarrow 0}\int_{\R^k}
\varphi\,d\mu_t=\int_{\R^k}\varphi\,d\mu_0, \qquad\forall \,
\varphi\in C^\infty_c(\R^k).
\end{equation}
Equivalently, \eqref{fk1} and \eqref{fk2} can be grouped by saying
that for every $T>0$ and $\varphi\in C^\infty_c([0,T]\times\R^k)$ we
have
\begin{equation}\label{eq:Gibbs:3}
  \int_{\R^k}\varphi_T\,d\mu_T=\int_{\R^k}\varphi\,d\mu_0
  +\int_0^T\int_{\R^k}\Big(\partial_t\varphi_t+\Delta\varphi_t-
    \langle\nabla V,\nabla\varphi_t\rangle \Big)\,d\mu_t\,dt.
\end{equation}

\begin{proposition}[Uniqueness and stability of FP solutions]\label{fpunique}
For any $\mu_0\in\Probabilities{\R^k}$, the Fokker-Planck equation
\eqref{eq:Gibbs:3} has a unique solution in the class of weakly
continuous maps $t\mapsto\mu_t\in\Probabilities{\R^k}$. If
$\mu_0\in\ProbabilitiesTwo{\R^k}$, then the unique solution
$[0,+\infty[ \, \ni t\mapsto\mu_t\in\ProbabilitiesTwo{\R^k}$ is
continuous. In addition, $\mu_t$ is stable: $\mu^n_0\to\mu_0$ in
$\ProbabilitiesTwo{\R^k}$ implies $\mu^n_t\to\mu_t$ in
$\ProbabilitiesTwo{\R^k}$ for all $t\geq 0$.
\end{proposition}
\begin{proof}
We consider first the case $\mu_0=\delta_x$: it only remains to
prove that $[0,\infty)\ni t\mapsto\nu_t^x\in\ProbabilitiesTwo{H}$ is
continuous. Taking \eqref{cami2} into account, it is enough to prove
that
  \begin{equation}\label{eq:Gibbs:3bis}
    \sup_{0\le t\le T}\Expectation{\|X_t(x)\|^2}
    \leq C(1+\|x\|^2)<+\infty,
  \end{equation}
  \begin{equation}\label{eq:Gibbs:3ter}
  \sup_{0\le t\le T}\Expectation{\|X_t(x)\|^2 1_{\{\|X_t(x)\|\geq R\}}}
  \leq\omega(R)(1+\|x\|^2)\quad\text{with}\quad
  \text{$\omega(R)\downarrow 0$ as $R\uparrow +\infty$,}
  \end{equation}
where $C$ and $\omega$ depend only on $T$ and $\sup\|\nabla V\|$. We
apply It\^o's formula to $\|X_t(x)\|$ and find that
\[
\|X_t\| \leq \|x\| + T\sup\|\nabla V\| + (k-1) \int_0^t
\frac1{\|X_s(x)\|} \, ds + \sqrt2 \, \hat B_t, \qquad \forall \,
t\in[0,T],
\]
where $\hat B$ is a standard Brownian motion in $\R$. We define now
the process $(b_t)$, unique non-negative solution of
\[
b_t \, =\,  \|x\| + T\sup\|\nabla V\| + (k-1) \int_0^t \frac1{b_s}
\, ds + \sqrt2 \, \hat B_t, \qquad \forall \, t\geq 0.
\]
Then $(b_t)$ is a Bessel process of dimension $k$, i.e. $(b_t,t\geq
0)$ is equal in law to $(\|b_0\cdot h+\sqrt2 \, W_t\|,t\geq 0)$,
where $h\in\R^k$ with $\|h\|=1$: see \cite{reyo}. By It\^o's
formula, $\|X_t\|\leq b_t$ for all $t\in[0,T]$, almost surely. Then
\eqref{eq:Gibbs:3bis} and \eqref{eq:Gibbs:3ter} follow from standard
Gaussian estimates.

 Existence of FP solutions, as we have seen, is provided
by \eqref{eq:dr1.5}. Uniqueness can for instance be obtained by a
classical duality argument: let $\mu^1_t,\,\mu^2_t$ be two weakly
continuous solutions of \eqref{fk1}, and let
$\sigma:=\mu^1_t-\mu_t^2$ be their difference, satisfying
  \begin{equation}
    \label{eq:Gibbs:6}
    \int_{\R^k}\varphi_T\,d\sigma_T=
    \int_0^T\int_{\R^k}
    \Big(\partial_t\varphi_t+\Delta\varphi_t \, - \, \langle
    \nabla V,\nabla\varphi_t\rangle \Big)\,d\sigma_t\,dt
  \end{equation}
  for every $T>0$ and $\varphi\in C^\infty_{\rm c}([0,T]\times \R^k)$.
  By a mollification technique, it is not difficult to check that
  \eqref{eq:Gibbs:6} holds even for every
  function $\varphi\in C([0,T]\times\R^k)$ with $\partial_t\varphi$, $\nabla\varphi$ and
  $\nabla^2\varphi$ continuous and bounded in $[0,T]\times\R^k$.
  For given $\psi\in C^\infty_c(\R^k)$ we consider
  the solution $\varphi_t$ of the time reversed (adjoint)
  parabolic equation
  \begin{equation}
    \label{eq:Gibbs:8}
    \partial_t\varphi_t+\Delta\varphi_t - \langle
    \nabla V,\nabla\varphi_t\rangle=0\quad\text{in $(0,T)\times\R^k$},
    \qquad\varphi_T=\psi.
  \end{equation}
  Standard parabolic regularity theory (it suffices to use the maximum principle
  \cite{john} and the fact that the first and second order spatial derivatives of $\varphi$
  solve an analogous equation)
  shows that $\varphi$ is sufficiently regular to be used as a test function in
  \eqref{eq:Gibbs:6}: this leads to $\int\psi\,d\sigma_T=0$. As $\psi$
  is arbitrary, we obtain that $\mu_T^1=\mu_T^2$.

  The representation $\mu_t=\int\law tx\,d\mu_0(x)$, given in \eqref{eq:dr1.5},
  and the uniform estimates \eqref{eq:Gibbs:3bis}, \eqref{eq:Gibbs:3ter}
  easily imply the stability property.
\end{proof}

Notice that the measure $\gamma$ provides a \emph{stationary}
solution of \eqref{fk} (and it can be actually shown that all
solutions $\mu_t$ weakly converge to $\gamma$ as $t\to +\infty$); it
is also natural to consider initial conditions
$\mu_0=\rho_0\gamma\in \Probabilities{H}$ with $\rho_0\in
L^2(\gamma)$. In this class of initial data, one can consider the
variational formulation of the FP equation induced by the symmetric
Dirichlet form
\begin{equation}\label{dirigamma}
\cE_\gamma(\rho,\eta):=\int_{\R^k}\langle\nabla\rho,\nabla\eta\rangle\,d\gamma,
\qquad \rho,\,\eta\in W^{1,2}_\gamma(\R^k),
\end{equation}
where $W^{1,2}_\gamma(\R^k)$ is the weighted Sobolev space
\begin{equation}\label{weighted}
W^{1,2}_\gamma(\R^k):=\left\{\rho\in L^2(\gamma)\cap W^{1,1}_{\rm
loc}(\R^k):\ \int_{\R^k}\|\nabla\rho\|^2\,d\gamma<+\infty\right\}.
\end{equation}
\begin{proposition}\label{prop:Fkpre} For every $\rho_0\in
L^2(\gamma)$:
\begin{enumerate}
\item  there exists a unique $\rho_\cdot\in W^{1,2}_{\rm
loc}\left(]0,+\infty[;L^2(\gamma)\right)$ such that
\begin{equation}
  \label{eq:Gibbs:16}
  \frac{d}{dt}\langle\rho_t,\eta\rangle_{L^2(\gamma)}+\cE_\gamma(\rho_t,\eta)=0, \quad
  \forall\, \eta\in W^{1,2}_\gamma(\R^k),\qquad  \lim_{t\downarrow0}\rho_t=\rho_0 \quad\text{in
  }L^2(\gamma);
\end{equation}
 if $\rho_{\rm min}\leq\rho_0\leq\rho_{\rm max}$,
  then $\rho_t$ satisfies the same uniform bounds;
\item if $\rho_0\geq 0$ and $\int \rho_0 \, d\gamma=1$, then $\mu_t=\rho_t\gamma \in\Probabilities{H}$ provides
  the unique solution of \eqref{fk1} starting from
  $\mu_0=\rho_0\gamma$;
\item if $\rho_0\in C_b(\R^k)$, then, for all $t\geq 0$, the function
\begin{equation}\label{tuttolega1}
P_t \rho_0(x) := \E(\rho_0(X_t(x))) = \int \rho_0 \, d\law tx,
\qquad \forall \, x\in\R^k,
\end{equation}
provides a continuous version of $\rho_t$, i.e. $P_t
\rho_0(x)=\rho_t(x)$ for $\gamma$-a.e. $x$; moreover $P_t$ acts on
${\rm Lip}_b(\R^k)$ and
 \begin{equation}\label{feller}
 [P_t \rho]_{{\rm Lip}(\R^k)}\leq [\rho]_{{\rm Lip}(\R^k)}\qquad \ t\geq 0,
 \quad \rho\in {\rm Lip}_b(\R^k);
 \end{equation}
\item $(P_t)_{t\geq 0}$ has an extension to a symmetric strongly
continuous semigroup in $L^2(\gamma)$.
\end{enumerate}
 \end{proposition}
\begin{proof}
  Existence of a unique solution of \eqref{eq:Gibbs:16}
  follows by the well-known theory of variational evolution equations,
  as well as the uniform lower and upper bounds on $\rho_t$, being
  $\cE_\gamma$ a Dirichlet form: this proves point 1. Now,
  for $\varphi\in C^\infty_c(\R^k)$ we can choose $\eta=\varphi\exp(V)$
  in \eqref{eq:Gibbs:16} and integrate by parts in space to obtain \eqref{fk1}:
  this shows that $\mu_t$ is the unique solution of the FP equation, as stated in point
  2. Continuity of $P_t\rho_0$ and \eqref{feller} follow from \eqref{eq:dr1}; in order
  to prove that $\rho_t=P_t\rho_0$ $\gamma$-a.e., we can reduce by
  linearity to the case $\rho_0\gamma\in\Probabilities{\R^k}$ and
  then point 3 follows from point 2. Point 4 follows from standard $L^2(\gamma)$ estimates for
  equation \eqref{eq:Gibbs:16}.
  \end{proof}

\medskip
Using the convexity inequality
$\cE_\gamma(\eta,\eta)\geq\cE_\gamma(\rho,\rho)+2\cE_\gamma(\rho,\eta-\rho)$,
it is not difficult to show that an equivalent formulation of \eqref{eq:Gibbs:16}
is (this kind of formulation first appeared in \cite{ben}, in connection with
nonlinear evolution problems in Banach spaces)
\begin{equation}\label{EVIdiri}
\frac{d}{dt} \, \frac{1}{2} \, \|\rho_t-\eta\|_{L^2(\gamma)}^2+
\frac 12 \, \cE_\gamma(\rho_t,\rho_t)\leq \frac 12 \,
\cE_\gamma(\eta,\eta), \qquad\forall \, \eta\in
W^{1,2}_\gamma(\R^k).
\end{equation}
We are going to show an analogous property of the solutions $\mu_t$
of the Fokker-Planck equation \eqref{eq:Gibbs:3}, obtained just
replacing $L^2$ norm with Wasserstein distance and $\frac 12
\cE_\gamma(\rho,\rho)$ with the relative Entropy functional
$\RelativeEntropy{\rho\gamma}{\gamma}$ with respect to $\gamma$.
This provides the key connection with the variational theory of
gradient flows in Wasserstein spaces. To this aim, let us first
establish the analogue of the convexity inequality for the relative
Entropy:

\begin{lemma}[Energy inequality]\label{lenergyi}
Let $\rho\in C^1(\R^k)\cap W^{1,2}_\gamma(\R^k)$, with $0<\rho_{\rm
min}\leq\rho\leq\rho_{\rm max}<+\infty$. Then:
\begin{equation}\label{energy}
\RelativeEntropy{\eta\gamma}{\gamma}\geq\RelativeEntropy{\rho\gamma}{\gamma}+
\int_{\R^k}\langle\nabla\rho,\tt-\ii\rangle\,d\gamma, \qquad \forall
\, \eta\gamma\in\ProbabilitiesTwo{\R^k},
\end{equation}
where $\tt$ is the optimal transport map between $\rho\gamma$ and
$\eta\gamma$.
\end{lemma}
\begin{proof} We just sketch the proof, referring to
\cite[Lemma~10.4.4, Lemma~10.4.5]{ags}
or to \cite{cmv} for more general results and detailed arguments.
Let $\tt=\nabla\phi$, with $\phi$ convex, and $u=\rho\exp(-V)$.
Defining $\mu^0=\rho\gamma$, $\mu^1=\eta\gamma$ and $\mu^t$ as in \eqref{eq:Gibbs:20},
taking \eqref{eq:Gibbs:19} into account
it suffices to bound from below $\frac{d}{dt^+}\RelativeEntropy{\mu^t}{\gamma}$ at $t=0$
with $\int\langle\nabla\rho,\tt-\ii\rangle\,d\gamma$. Now, a computation based on
the area formula (that provides an explicit expression for the density of $\mu^t$
with respect to $\Leb{k}$, see also the proof of Lemma~\ref{slopeentro} in the Appendix) gives
$$
\frac{d}{dt^+}\RelativeEntropy{\mu^t}{\gamma}\biggr\vert_{t=0}
=-\int_{\R^k}{\rm trace\,}(\nabla^2\phi-I)u(x)\,dx+
\int_{\R^k}\langle\nabla V,\tt-\ii\rangle\,\rho\,d\gamma,
$$
where $\nabla^2\phi$ is the Alexandrov pointwise second derivative
of $\phi$ and $I$ is the identity matrix. By the convexity of $\phi$
we can bound the matrix of absolutely continuous measures
$\nabla^2\phi \, \Leb{k}$ from above with the distributional
derivative of $\nabla\phi=\tt$ (which, in general, is a measure, by
the convexity of $\phi$) to obtain
$$
\frac{d}{dt^+}\RelativeEntropy{\mu^t}{\gamma}\biggr\vert_{t=0}
\geq -\langle\nabla\cdot (\tt-\ii),u\rangle+
\int_{\R^k}\langle\nabla V,\tt-\ii\rangle\,\rho\,d\gamma.
$$
Finally, we integrate by parts to obtain \eqref{energy}: although
$u$ is not compactly supported, this formal integration can be
justified by approximation of $u$ with $u\psi_R$, where $\psi_R\in
C^\infty_c(\R^k)$, $0\leq\psi_R\leq 1$, $\|\nabla\psi_R\|\leq 1$,
$\psi_R\uparrow 1$ and $\nabla\psi_R\to 0$ as $R\to+\infty$.
\end{proof}

\begin{proposition}\label{abs}
For all $\mu_0\in\ProbabilitiesTwo{\R^k}$ the solution $(\mu_t)$ of
the FP equation \eqref{fk}, characterized in
Proposition~\ref{fpunique}, satisfies the family of variational
evolution inequalities
\begin{equation}\label{EVI'}
\frac{d}{dt}\frac{1}{2} W_2^2(\mu_t,\nu)+
\RelativeEntropy{\mu_t}{\gamma}\leq\RelativeEntropy{\nu}{\gamma}
\end{equation}
in the sense of distributions in $]0,+\infty[$, for all
$\nu\in\ProbabilitiesTwo{\R^k}$.
\end{proposition}
\begin{proof}
First of all, we notice that the variational evolution inequalities
\eqref{EVI'} are stable with respect to pointwise convergence in
$\ProbabilitiesTwo{\R^k}$: indeed, if $\mu^n_t\to\mu_t$ for all $t$,
then $\frac{d}{dt}W_2^2(\mu^n_t,\nu)\to\frac{d}{dt}W_2^2(\mu_t,\nu)$
in the sense of distributions, and the lower semicontinuity of
$\RelativeEntropy{\cdot}{\gamma}$ allows to pass to the limit as
$n\to\infty$.

Thanks to this remark and to the stability properties of solutions
to FP equations, we need only to show the property when
$\mu_0=\rho_0\gamma$ with $\rho_0\in C^2_b(\R^k)$,
$\rho_0\geq\rho_{{\rm min}}>0$. Then, we know from
Proposition~\ref{prop:Fkpre} that $\mu_t=\rho_t\gamma$, with
$\rho_t$ smooth, $\rho_t\geq \rho_{{\rm min}}$. In addition, writing
$u_t=\rho_t\exp(-V)$, since $u_0=\rho_0\exp(-V)$ belongs to
$C^2_b(\R^k)$ as well, standard parabolic regularity theory for the
FP equation whose drift is bounded, together will all its
derivatives, gives $\partial_t u,\,\nabla u,\,\nabla^2u\in
C_b\left([0,T]\times\R^k\right)$ for all $T>0$.

We are interested in getting pointwise bounds for the velocity field
$\vv_t:=-(\nabla\rho_t)/\rho_t$; it appears in a natural way in
this problem because, by \eqref{fk}, $\mu_t$ solve the classical
\emph{continuity equation}
$$
\frac{d}{dt}\mu_t+\nabla\cdot (\vv_t\mu_t)=0 \qquad\text{in
$]0,+\infty[\times\R^k$,}
$$
describing the evolution of a time-dependent mass distribution $\mu_t$ under the
action of a velocity field $\vv_t$.
Since $\rho_t=u_t\exp(V)$, and $\exp(V)$ is not bounded above, we can not
use directly the $C^2_b$ bounds on $u_t$. However, we can use the fact that
$\rho_t$ solve the linear PDE $\partial_t\rho_t=\Delta\rho_t- \langle\nabla\rho_t,\nabla V\rangle$,
to obtain
$$
\partial_t\nabla\rho=\Delta\nabla\rho-\nabla^2V\nabla\rho-\nabla^2\rho\nabla V.
$$
Taking the scalar product with $\nabla\rho$ and using the identity
$\Delta\|\nabla\rho\|^2=2\langle\Delta\nabla\rho,\nabla\rho\rangle+2\|\nabla^2\rho\|^2$,
we can neglect the nonnegative terms $\|\nabla^2\rho\|^2$ and $\langle\nabla^2 V\nabla\rho,\nabla\rho\rangle$
to get
$$
\partial_t\|\nabla\rho_t\|^2\leq
\Delta \|\nabla\rho_t\|^2-\langle\nabla V,\nabla\|\nabla\rho_t\|^2\rangle.
$$
The classical maximum principle can now be applied, thanks to the fact that
$\|\nabla\rho_t\|^2$ grows at most exponentially \cite[Page 216]{john} to obtain
$\|\nabla\rho_t\|^2\leq\|\nabla\rho_0\|^2$ for all $t\geq 0$.
By the uniform lower bound on $\rho_t$, $\|\vv_t\|$ is uniformly bounded.

Now, let us show first that $t\mapsto\mu_t$ is a Lipschitz map in $[0,+\infty[$
with values in $\ProbabilitiesTwo{\R^k}$. Let $0\leq t_1\leq t_2<+\infty$;
the smoothness and the just proved boundedness of $\vv_t$ ensure
the existence of a unique flow $Y(t,x):[0,+\infty[\times\R^k\to\R^k$
associated to $\vv_t$, i.e. $Y(0,x)=x$ and $\frac{d}{dt}Y(t,x)=\vv_t(Y(t,x))$
in $[0,+\infty[$. Then, the method of characteristics
(see e.g. \cite[Proposition~8.1.8]{ags}) for solutions to the continuity
equation shows that $\mu_t$ is given by $Y(t,\cdot)_\#\mu_0$ for all
$t\in [0,+\infty[$. Therefore we can use the coupling
$(Y(t_2,\cdot),Y(t_1,\cdot))_\#\mu_0$ to estimate $W_2(\mu_{t_1},\mu_{t_2})$ as follows:
$$
W_2^2(\mu_{t_1},\mu_{t_2})\leq\int_{\R^k}\|Y(t_2,x)-Y(t_1,x)\|^2\,d\mu_0=
\int_{\R^k}\|\int_{t_1}^{t_2} \vv_t(Y(t,x))\,dt\|^2\,d\mu_0
\leq C(t_2-t_1)^2.
$$
This proves that $t\mapsto\mu_t$ is Lipschitz in $[0,+\infty[$.

To conclude the proof, it suffices to check \eqref{EVI'} at any differentiability
point $\bar t\in [0,+\infty[$ of the map $t\mapsto W_2^2(\mu_t,\nu)$. Let $\tt$ be the optimal
transport map between $\mu_{\bar t}$ and $\nu$, and let $h>0$; since
$\Sigma=(Y(\bar t+h,\cdot),\tt\circ Y(\bar t,\cdot))_\#\mu_0$ is a
coupling between $\mu_{\bar t+h}$ and $\nu$, we can estimate
(using the identity $\|a\|^2-\|b\|^2=\langle a+b,a-b\rangle$):
\[
\begin{split}
&\frac{W_2^2(\mu_{\bar t+h},\nu)-W_2^2(\mu_{\bar t},\nu)}{h} \leq
\frac{1}{h}\int_{\R^k}\|\tt(Y(\bar t,x))-Y(\bar t+h,x)\|^2\,d\mu_0-
\frac{1}{h}\int_{\R^k}\|\tt(y)-y\|^2\,d\mu_{\bar t}
\\&= \frac{1}{h}\int_{\R^k}\Big[ \|\tt(Y(\bar t,x))-
Y(\bar t+h,x)\|^2\,- \|\tt(Y(\bar t,x))-Y(\bar t,x)\|^2\Big] d\mu_0\\
&= -\int_{\R^k}\langle 2\tt(Y(\bar t,x))-Y(\bar t+h,x)-Y(\bar t,x),
\frac{1}{h}\left(Y(\bar t+h,x)-Y(\bar t,x) \right) \rangle\,d\mu_0
\\ &\longrightarrow-2\int_{\R^k} \langle\tt(Y(\bar t,x))-Y(\bar t,x),
\vv_{\bar t}(Y(\bar t,x))\rangle\,d\mu_0=
-2\int_{\R^k}\langle\tt(y)-y,\vv_{\bar t}(y)\rangle\,d\mu_{\bar t},
\end{split}
\]
as $h\downarrow 0$, by dominated convergence.
From the energy inequality \eqref{energy} we obtain \eqref{EVI'}.
\end{proof}

Starting from \eqref{fk}, we have derived a new relation
\eqref{EVI'} satisfied by solutions to FP equations, at least when
$H=\R^k$ and $V$ is smooth, with $\nabla V$ bounded and Lipschitz.
The idea of \cite{ags} is to consider \eqref{EVI'} as the
\emph{definition} of a differential equation in a space of
probability measures endowed with the Wasserstein distance even when
either $V$ is less regular or $H$ is infinite-dimensional: more
precisely, $(\mu_t)_{t\geq 0}$ is said to solve the \emph{gradient
flow} of the functional $\RelativeEntropy{\cdot}{\gamma}$; the
choice of $\delta_x$ as initial measure is the one that provides the
link with the laws $\law tx$ of the solution $X_t(x)$ of the SDE
\eqref{gs3}.

Notice that \eqref{EVI'} is defined only in terms of
the Wasserstein distance and the relative entropy,
namely of objects which make perfectly sense on
an arbitrary Hilbert space $H$. Motivated by Proposition~\ref{abs},
we set the following:
\begin{definition}[Gradient flows]\label{defgrfl} $ $ \\
Let $F:\ProbabilitiesTwo{H}\mapsto[0,+\infty]$ and set
$D(F)=\{F<+\infty\}$. We say that a continuous map $\mu_t: \,
]0,+\infty[ \, \mapsto\ProbabilitiesTwo{H}$ is a gradient flow of
$F$ if
\begin{equation}\label{EVI}
\frac{d}{dt}\frac{1}{2}W_2^2(\mu_t,\nu)+F(\mu_t)\leq F(\nu)
\quad\text{in the sense of distributions in $]0,+\infty[$, for all $\nu\in D(F)$.}
\end{equation}
We say that $\mu_t$ starts from $\mu_0$ if $\mu_t\to\mu_0$ in
$\ProbabilitiesTwo{H}$ as $t\downarrow 0$.
\end{definition}
The terminology ``gradient flow'' can be justified, by appealing to
Otto's formal differential calculus on $\ProbabilitiesTwo{\R^k}$;
since this calculus will not play a significant role in our paper we
will not discuss this issue, and refer to \cite{jko}, \cite{ags},
\cite{vil} for much more heuristics on this subject. Here we just
point out that existence of gradient flows can be obtained  (see
Section~\ref{exunst}) by the analogue in the Wasserstein setting of
the Euler scheme for the approximation of gradient flows $\dot
x(t)=-\nabla F(x(t))$: namely, given a time step $\tau>0$, we build
a sequence $(x_k)$ by minimizing
$$
y\mapsto \frac{1}{2\tau}\|y-x_k\|^2+F(y)
$$
recursively (i.e., given $x_k$, we choose $x_{k+1}$ among the minimizers
of the variational problem above). Looking at the discrete Euler equation,
$(x_{k+1}-x_k)/\tau=-\nabla F(x_{k+1})$, it is clear that $x_k\sim x(k\tau)$.

Notice also that \eqref{EVI} implies that
$\mu_t\equiv\mu_0\in\ProbabilitiesTwo{H}$ is a constant gradient
flow if and only if $\mu_0$ minimizes $F$; in the case
$F(\mu)=\RelativeEntropy{\mu}{\gamma}$, since $t\mapsto t\ln t$ is
strictly convex, the unique minimizer of $F$ is $\mu=\gamma$. So,
from the gradient flow viewpoint, we easily see that the unique
invariant measure in $\ProbabilitiesTwo{H}$ is $\gamma$ if
$\gamma\in\ProbabilitiesTwo{H}$: see Theorem \ref{main4} and its
proof in section \ref{process}.

In the next sections we are going to adopt the ``gradient flow'' point of view,
and prove that the results of Theorems \ref{main1}, \ref{main2} and \ref{main3}
are relatively easy consequences of this approach.

\section{Implicit Euler scheme}
\label{scheme}

In this section we construct a discrete approximation
of the gradient flow. Such construction is based on
the following convexity property, which is a stronger
version of the one given in Proposition~\ref{prop:displ_conv}.

\begin{definition}
We say that $F:\ProbabilitiesTwo{H}\to [0,+\infty]$ is
strongly displacement convex if for any
$\bar\mu,\nu_0,\nu_1\in\ProbabilitiesTwo{H}$
there exists a continuous curve $\nu:[0,1]\mapsto\ProbabilitiesTwo{H}$
such that $\nu_{\vert t=i}=\nu_i$, $i=0,1$, such that:
\begin{equation}\label{convdist'}
\begin{cases}
\text{$W_2^2(\nu_t,\bar\mu)\leq (1-t) W_2^2(\nu_0,\bar\mu)+t W_2^2(\nu_1,\bar\mu)
-t(1-t)W_2^2(\nu_0,\nu_1)$}&\\
\text{$F(\nu_t)\leq (1-t)F(\nu_0)+tF(\nu_1)$}&
\end{cases}
\qquad\forall t\in [0,1].
\end{equation}
\end{definition}
In this paper we consider only the case
$F(\mu):=\RelativeEntropy{\mu}{\gamma}$, where $\gamma$ is a
log-concave probability measure on $H$. In fact the following
results are true for much more general classes of strongly
displacement convex functionals $F:\ProbabilitiesTwo{H}\to
(-\infty,+\infty]$ with essentially the same proof (under suitable
lower semicontinuity and coercivity assumptions). However at one
point we shall take advantage of a particular feature of the
relative entropy functional $\RelativeEntropy{\cdot}{\gamma}$,
namely the entropy inequality \eqref{entrine}, in order to simplify
the proof. For more general cases, see Chapter~4 in \cite{ags}.

In order to build gradient flows, we use an implicit Euler scheme, at
least when $\bar\mu\in D(F)$; then, suitable Cauchy-type estimates provide
existence up to initial data in $\overline{D(F)}$, as in the Hilbertian
theory. The scheme can be described as follows: given a time step
$\tau>0$, we define a ``discrete'' solution $\mu^k_\tau$ setting
$\mu^0_\tau=\bar\mu$ and, given $\mu^k_\tau$, we choose $\mu^{k+1}_\tau$
as the unique minimizer of
\begin{equation}\label{euler1}
\nu \ \mapsto \
\RelativeEntropy{\nu}{\gamma}+\frac{1}{2\tau}W_2^2(\nu,\mu^k_\tau).
\end{equation}
The fact that this is possible is a consequence of the following
\begin{proposition}\label{exdi}
Let $\tau>0$.
For all $\mu\in\ProbabilitiesTwo{H}$ there exists
a unique $\mu_\tau\in\ProbabilitiesTwo{H}$ such that:
\begin{equation}\label{euler15}
\RelativeEntropy{\mu_\tau}{\gamma}+\frac{1}{2\tau}W_2^2(\mu_\tau,\mu) \, \leq \,
\RelativeEntropy{\nu}{\gamma}+\frac{1}{2\tau}W_2^2(\nu,\mu)\qquad
\forall \ \nu\in\ProbabilitiesTwo{H}.
\end{equation}
\end{proposition}
The existence part of this result is rather standard and relies on
tightness and lower semicontinuity arguments. The uniqueness
statement is based on the strong displacement convexity of the
relative Entropy functional, proved in Proposition~\ref{convdist}.
\begin{proof} \emph{Existence}. Let $\nu_k=f_k\gamma\in\ProbabilitiesTwo{H}$
be such that:
\[
\lim_{k\to\infty}
\left[ F(\nu_k)+\frac{1}{2\tau}W_2^2(\nu_k,\mu) \right]
 = \ \inf_{\nu\in\ProbabilitiesTwo{H}}\left\{F(\nu)+\frac{1}{2\tau}W_2^2(\nu,\mu)
\right\} < \, +\infty.
\]
In particular we have that $(\nu_k)$ is bounded
in $\ProbabilitiesTwo{H}$ and
$\limsup_k\RelativeEntropy{\nu_k}{\gamma}<\infty$.
By using first the inequality $t\ln t\geq -e^{-1}$ and then
Jensen inequality we get
\begin{equation}\label{entrine}
\frac{1}{e}\gamma(H\setminus E)+\RelativeEntropy{\nu_k}{\gamma}\geq
\int_E f_k\ln f_k\,d\gamma\geq\nu_k(E)\ln\frac{\nu_k(E)}{\gamma(E)}
\qquad \forall \ E\in\BorelSets{H},
\end{equation}
so that $\gamma(E)\to 0$ implies $\sup_k\nu_k(E)\to 0$. It
follows that $(\nu_k)$ is tight in $H$, so that we can extract
a subsequence, that we can still denote by $(\nu_k)$, converging weakly
to some $\mu_\tau\in\ProbabilitiesTwo{H}$.

If we prove that both $F$ and $W_2^2(\cdot,\mu)$ are lower
semicontinuous with respect to weak convergence, then we have that
$\mu_\tau$ realizes the minimum in \eqref{euler15}: indeed by lower
semicontinuity:
\[
F(\mu_\tau)+\frac{1}{2\tau}W_2^2(\mu_\tau,\mu) \, \leq \,
\liminf_{k\to\infty}\left[ F(\nu_k)+\frac{1}{2\tau}W_2^2(\nu_k,\mu) \right]
= \inf_{\nu\in\ProbabilitiesTwo{H}}\left\{F(\nu)+\frac{1}{2\tau}W_2^2(\nu,\mu)
\right\}.
\]
A nice representation of the relative entropy functional is provided by
the \emph{duality formula} (see for instance Lemma~9.4.4 of \cite{ags}):
\begin{equation}\label{duality}
\RelativeEntropy{\mu}{\gamma}=\sup\left\{\int_H S\,d\mu-\int_H (e^S-1)\,d\gamma:\
 S\in C_b(H)\right\}.
\end{equation}
This formula immediately implies that
$\RelativeEntropy{\cdot}{\gamma}$ is sequentially lower
semicontinuous with respect to the weak convergence.

Let now $\Sigma_k\in\Gamma_o(\nu_k,\mu)$ and assume with no
loss of generality that $W_2(\nu_k,\mu)$ converges to some limit;
since $(\nu_k)$ is tight in $H$, $(\Sigma_k)$ is tight in
$H\times H$ and we can assume that $(\Sigma_k)$ converges
weakly to $\Sigma\in\Probabilities{H\times H}$.
Obviously $\Sigma\in\Gamma(\nu,\mu)$ and the weak
convergence of $\Sigma_n$ gives
\begin{equation}\label{lsc}
\int_{H\times H}\|x-y\|^2\,d\Sigma\leq\liminf_{k\to\infty}\int_{H\times H}
\|x-y\|^2\,d\Sigma_k.
\end{equation}
Bounding $W_2^2(\mu,\nu)$ from above with $\int\|x-y\|^2\,d\Sigma$
we obtain the lower semicontinuity property. A similar argument also
proves the \emph{joint} lower semicontinuity of $(\mu,\nu)\mapsto
W_2(\mu,\nu)$ (we will use this fact at the end of the proof of
Proposition~\ref{convdist}).

\medskip\emph{Uniqueness.}
Suppose that $\tilde \mu_\tau\neq\mu_\tau$ realize the minimum in \eqref{euler15},
denoted by $m$. Let $\nu_t$ be a curve between $\tilde\mu_\tau$ and $\mu_\tau$
given by Proposition \ref{convdist} below, with the choice $\bar\mu=\mu$.
Then we obtain:
\[
F(\nu_{1/2})+\frac{1}{2\tau}W_2^2(\nu_{1/2},\mu) \, \leq
m -\frac{1}{4}\, W^2_2(\tilde\mu_\tau,\mu_\tau) \, < \, m,
\]
which is a contradiction.
\end{proof}

\begin{proposition}\label{convdist}
Let $\gamma\in\Probabilities{H}$ be log-concave.
Then the functional $\RelativeEntropy{\cdot}{\gamma}:\ProbabilitiesTwo{H}\to [0,+\infty]$ is
strongly displacement convex.
\end{proposition}
\begin{proof}
As shown in \cite{ags}, it is often enough to build the
interpolating curves only for a dense subset $\mathscr D$ of
measures $\bar\mu$. In the case of the relative entropy functional
(but also for more general classes of functionals, see \cite{ags})
the set $\mathscr D$ is made by finite convex combinations of
non-degenerate Gaussian measures; as $\overline{\mathscr D}$ is
easily seen to contain finite convex combinations of Dirac masses,
$\mathscr D$ is dense in $\ProbabilitiesTwo{H}$. Moreover, any
measure in $\mathscr D$ vanishes on the class $\GBorelSets{H}$ of
Gaussian null sets; hence, for any $\bar\mu\in\mathscr D$ and
$\nu_0,\,\nu_1\in\ProbabilitiesTwo{H}$ we can find optimal transport
maps $\rr_i$ between $\bar\mu$ and $\nu_i$, $i=0,1$. We define
\begin{equation}\label{definterpola}
\nu_t:=\left((1-t)\rr_0+t\rr_1\right)_\#\mu.
\end{equation}
Let us first check the Lipschitz continuity of
$t\mapsto\nu_t\in\ProbabilitiesTwo{H}$: for $s,\,t\in [0,1]$, the
coupling
$$
\Sigma_{st}:=\left((1-s)\rr_0+s\rr_1,(1-t)\rr_0+t\rr_1\right)_\#\bar\mu
$$
belongs to $\Gamma(\nu_s,\nu_t)$, so that
\begin{eqnarray*}
W_2^2(\nu_s,\nu_t)&\leq&\int_{H\times H}\|x-y\|^2\,d\Sigma_{st}=
|t-s|^2\int_H\|\rr_1-\rr_0\|^2\,d\bar\mu\\
&\leq&2|t-s|^2 \left(W_2^2(\nu_1,\bar\mu)+W_2^2(\nu_0,\bar\mu)\right).
\end{eqnarray*}
Now, let us check the convexity of $t\mapsto W_2^2(\nu_t,\bar\mu)$:
\begin{eqnarray*}
W_2^2(\nu_t,\bar\mu)&\leq&
\int_H \|(1-t)(\rr_0-\ii)+t(\rr_1-\ii)\|^2\,d\bar\mu
\\
&=& (1-t)\int_H\|\rr_0-\ii\|^2\,d\bar\mu
+t\int_H\|\rr_1-\ii\|^2\,d\bar\mu
-t(1-t)\int_H\|\rr_0-\rr_1\|^2\,d\bar\mu
\\ &\leq& (1-t)W_2^2(\nu_0,\bar\mu)+t
W_2^2(\nu_1,\bar\mu)-t(1-t)W_2^2(\nu_0,\nu_1).
\end{eqnarray*}
In the last inequality we used the fact that $(\rr_0,\rr_1)_\#\bar\mu\in\Gamma(\nu_0,\nu_1)$.
For the convexity of $t\mapsto F(\nu_t)$, achieved through a finite-dimensional
approximation, we refer to \cite[Theorem 9.4.11]{ags}.

Having built the interpolating curves when $\bar\mu\in\mathscr D$,
in the general case, we can approximate any
$\bar\mu\in\ProbabilitiesTwo{H}$ by measures $\bar\mu^n\in\mathscr
D$; notice that the interpolating curves $t\mapsto\nu^n_t$ between
$\nu_0$ and $\nu_1$ are equi-Lipschitz and, for $t$ fixed, the same
tightness argument used in the existence part of
Proposition~\ref{exdi} shows that $(\nu^n_t)$ is tight. Therefore,
thanks to a diagonal argument, we can assume that $\nu^n_t\to\nu_t$
weakly for all $t\in [0,1]\cap\Q$, with $t\mapsto\nu_t$ Lipschitz in
$[0,1]\cap\Q$. Passing to the limit as $n\to\infty$ in the convexity
inequalities relative to $\nu^n_t$, and using the weak lower
semicontinuity of $F(\cdot)$ and $W_2^2(\cdot,\cdot)$ we get
\begin{equation}
\begin{cases}
\text{$W_2^2(\nu_t,\bar\mu)\leq
(1-t) W_2^2(\nu_0,\bar\mu)+t W_2^2(\nu_1,\bar\mu)
-t(1-t)W_2^2(\nu_0,\nu_1)$}&\\
\text{$F(\nu_t)\leq (1-t)F(\nu_0)+tF(\nu_1)$}&
\end{cases}
\end{equation}
for all $t\in [0,1]\cap\Q$. By a density argument, based on the
completeness of $\ProbabilitiesTwo{H}$, we can obtain a Lipschitz
curve $\nu_t$ defined in the whole of $[0,1]$, still retaining the
inequalities above.
\end{proof}

\noindent We prove now an important estimate which plays a key role
in the sequel, see the proof of Theorem \ref{tgflow} below.

\begin{proposition}\label{4.4}
Let $F:\ProbabilitiesTwo{H}\to [0,+\infty]$ be
strongly displacement convex, let $\bar\mu\in\ProbabilitiesTwo{H}$
and let $\mu_\tau$ be a minimizer of
$$
\mu\mapsto F(\mu)+\frac{1}{2\tau}W_2^2(\mu,\bar\mu).
$$
Then
\begin{equation}\label{discrevi}
W_2^2(\mu_\tau,\nu)-W_2^2(\bar\mu,\nu) \, \leq \, 2 \, \tau \,
[F(\nu)-F(\mu_\tau)], \qquad\forall \,\nu\in D(F).
\end{equation}
\end{proposition}
\begin{proof} Let $\nu_0=\mu_\tau$, $\nu_1=\nu$ and consider the interpolating
curve $\nu_t:[0,1]\to\ProbabilitiesTwo{H}$ along which
\eqref{convdist'} holds. The minimality of $\mu_\tau$ and \eqref{convdist'} give
\begin{eqnarray*}
& & F(\mu_\tau)+\frac{1}{2\tau} \, W_2^2(\mu_\tau,\bar\mu) \leq
F(\nu_t)+\frac{1}{2\tau} \, W_2^2(\nu_t,\bar\mu)
\\ & & \leq
(1-t)\left[F(\mu_\tau)+\frac{1}{2\tau} \, W_2^2(\mu_\tau,\bar\mu)\right]+
t\left[F(\nu)+\frac{1}{2\tau} \, W_2^2(\nu,\bar\mu)\right] -
\, t(1-t) \, \frac{1}{2\tau} \, W_2^2(\mu_\tau,\nu).
\end{eqnarray*}
Subtracting $F(\mu_\tau)+W_2^2(\mu_\tau,\bar\mu)/2\tau$ from
the left hand side of the first inequality and from the right hand side
of the second inequality, and dividing by $t>0$ we obtain:
\begin{eqnarray*}
F(\nu) - F(\mu_\tau) & \geq &
\frac{1}{2\tau} \, W_2^2(\mu_\tau,\bar\mu) \, - \,
\frac{1}{2\tau} \, W_2^2(\nu,\bar\mu) \, + \,
\frac{1-t}{2\tau} \, W_2^2(\mu_\tau,\nu)
\\ & \geq &
\frac{1}{2\tau} \left[ (1-t) \,  W_2^2(\mu_\tau,\nu) \, - \,
W_2^2(\nu,\bar\mu) \right].
\end{eqnarray*}
Letting $t\downarrow 0$ we have
\begin{equation}
  \label{discrevi2}
  W_2^2(\mu_\tau,\nu)- W_2^2(\nu,\bar\mu)+ W_2^2(\mu_\tau,\bar\mu)
   \le 2\tau\Big(F(\nu) - F(\mu_\tau)\Big),
\end{equation}
which yields \eqref{discrevi} by neglecting the
nonnegative term $W_2^2(\mu_\tau,\bar\mu)$.
\end{proof}

\section{Existence and uniqueness of gradient flows}
\label{exunst}

In this section we prove existence and uniqueness
of gradient flows and convergence of the approximations
\eqref{euler1}. Again the results of this section
hold for more general classes of strongly displacement convex functionals,
but we are only interested here in the case
$F(\cdot)=\RelativeEntropy{\cdot}{\gamma}$,
where we consider a fixed log-concave probability
measure $\gamma$ on $H$.

We go back to the sequence $(\mu^k_\tau)_k$ defined recursively
by \eqref{euler1} with $\mu^0_\tau=\overline\mu\in\ProbabilitiesTwo{H}$,
the existence of $(\mu^k_\tau)_k$ being granted by Proposition \ref{exdi}.
We shall denote the ``discrete'' semigroup induced by $\mu^k_\tau$ by
$\Semit{\bar\mu}{t}{\tau}$, precisely
\begin{equation}\label{recu}
\Semit{\bar\mu}{t}{\tau}:=\mu^{k+1}_\tau\qquad\forall t\in (k\tau,(k+1)\tau].
\end{equation}

\begin{theorem}[Existence and uniqueness of gradient flows]\label{tgflow}
For any $\bar\mu\in \ProbabilitiesTwo{K}$
there exists a unique gradient flow starting from $\bar\mu$. The induced
semigroup $\Semi{\bar\mu}{t}$ satisfies
\begin{equation}\label{sat1}
W_2(\Semi{\bar\mu}{t},\Semi{\bar\mu}{s})\leq
\sqrt{2F(\bar\mu)}\sqrt{|t-s|}, \qquad t,\,s\geq 0,\,\,\bar\mu\in\ProbabilitiesTwo{K}
\end{equation}
and the following properties:
\begin{itemize}
\item[(i)] {\bf (Uniform discrete approximation)}
$W_2(\Semi{\bar\mu}{t},\Semit{\bar\mu}{t}{\tau})\leq C\sqrt{\tau F(\bar\mu)}$
if $\bar\mu\in D(F)$, with $C=2(2\sqrt 2+1)$;
\item[(ii)]{\bf (Contractivity)}
$W_2(\Semi{\bar\mu}{t},\Semi{\bar\nu}{t})\leq W_2(\bar\mu,\bar\nu)$;
\item[(iii)] {\bf (Regularizing effect)}
$F(\Semi{\bar\mu}{t})\leq\inf\limits_{\nu\in D(F)}\frac{1}{2t}W_2^2(\bar\mu,\nu)+F(\nu)<+\infty$
for all $t>0$, $\bar\mu\in\overline{D(F)}$.
\end{itemize}
\end{theorem}
\begin{proof} We first sketch the proof of uniqueness of gradient flows, referring
to \cite[Corollary~4.3.3]{ags} for all technical details: if  $\mu^1(t)$, $\mu^2(t)$
are gradient flows starting from $\bar\mu$, setting $\nu=\mu^1(t)$
into
$$
\frac{d}{dt}\frac{1}{2} W_2^2(\mu^2(t),\nu)\leq F(\nu)-F(\mu^2(t))
$$
and $\nu=\mu^2(t)$ into
$$
\frac{d}{dt}\frac{1}{2} W_2^2(\mu^1(t),\nu)\leq F(\nu)-F(\mu^1(t)),
$$
one obtains that $\frac{d}{dt}W_2^2(\mu^1(t),\mu^2(t))\leq 0$,
whence the identity of $\mu^1$ and $\mu^2$ follows.

In order to show existence of gradient flows,
we consider first the case when $\bar\mu\in D(F)$.
Notice that $\mu^{k+1}_\tau$ satisfies
\begin{equation}\label{eeuler1}
F(\mu^{k+1}_\tau)+\frac{1}{2\tau}W_2^2(\mu^{k+1}_\tau,\mu^k_\tau)
\leq F(\nu)+\frac{1}{2\tau}W_2^2(\nu,\mu^k_\tau)
\qquad\forall \nu\in D(F)
\end{equation}
and, choosing in particular $\nu=\mu^k_\tau$, we obtain that
$F(\mu^{k+1}_\tau)\leq F(\mu^k_\tau)$ and
\begin{equation}\label{eeee}
W_2(\mu^{k+1}_\tau,\mu^k_\tau)\leq\sqrt{2\tau
[F(\mu^k_\tau)-F(\mu^{k+1}_\tau)]}.
\end{equation}
This inequality easily leads to the discrete $C^{1/2}$ estimate
\begin{equation}\label{eeeuler1}
W_2(\Semit{\bar\mu}{t}{\tau},\Semit{\bar\mu}{s}{\tau})\leq
\sqrt{2F(\bar\mu)}\sqrt{|t-s+\tau|}.
\end{equation}
Moreover a crucial role is played by the formula
\begin{equation}\label{discrevibis}
W_2^2(\Semit{\bar\mu}{(k+1)\tau}{\tau},\nu)-
W_2^2(\Semit{\bar\mu}{k\tau}{\tau},\nu)\leq
2\tau [F(\nu)-F(\Semit{\bar\mu}{(k+1)\tau}{\tau})]
\end{equation}
for all $\nu\in D(F)$, which follows from Proposition \ref{4.4}.

\medskip\noindent
\emph{Proof of (i).} We start proving the estimate
\begin{equation}\label{dyadic1}
W_2^2(\Semit{\bar\mu}{t}{\tau},\Semit{\bar\nu}{t}{{\frac{\tau}{2}}})-
W_2^2(\bar\mu,\bar\nu)\leq 2\tau F(\bar\nu)
\end{equation}
for all $\tau>0$ and all times $t$ that are integer multiples of $\tau$. To this aim,
from \eqref{discrevibis} we obtain the inequalities
\begin{eqnarray}\label{beni1}
&W_2^2(\Semit{\bar\nu}{{\frac{\tau}{2}}}{{\frac{\tau}{2}}},\theta)-W_2^2(\bar\nu,\theta)&\leq
\tau\left[F(\theta)-F(\Semit{\bar\nu}{{\frac{\tau}{2}}}{{\frac{\tau}{2}}})\right],
\\
\label{beni2}
&W_2^2(\Semit{\bar\nu}{\tau}{{\frac{\tau}{2}}},\theta)-
W_2^2(\Semit{\bar\nu}{{\frac{\tau}{2}}}{{\frac{\tau}{2}}},\theta)&\leq
\tau\left[F(\theta)-F(\Semit{\bar\nu}{\tau}{{\frac{\tau}{2}}})\right],
\end{eqnarray}
for all $\theta\in D(F)$, whose sum gives
\begin{equation}\label{beni12}
W_2^2(\Semit{\bar\nu}{\tau}{\frac{\tau}{2}},\theta)-W_2^2(\bar\nu,\theta)\leq
\tau\left[2F(\theta)-F(\Semit{\bar\nu}{{\frac{\tau}{2}}}{{\frac{\tau}{2}}})-
F(\Semit{\bar\nu}{\tau}{{\frac{\tau}{2}}})\right]
\end{equation}
for all $\theta\in D(F)$. Still from \eqref{discrevibis} we get
\begin{equation}\label{beni3}
W_2^2(\Semit{\bar\mu}{\tau}{\tau},\theta)-W_2^2(\bar\mu,\theta)\leq
2\tau \left[F(\theta)-F(\Semit{\bar\mu}{\tau}{\tau})\right]\qquad
\forall\theta\in D(F).
\end{equation}
Setting $\theta=\Semit{\bar\mu}{\tau}{\tau}$ in \eqref{beni12}
and $\theta=\bar\nu$ in \eqref{beni3}, we can add
the resulting inequalities to obtain
\begin{eqnarray}\label{dyadic2}\nonumber
W_2^2(\Semit{\bar\mu}{\tau}{\tau},\Semit{\bar\nu}{\tau}{{\frac{\tau}{2}}})
-W_2^2(\Semit{\bar\mu}{0}{\tau},\Semit{\bar\nu}{0}{\frac{\tau}{2}})&\leq&
\tau\left(2F(\bar\nu)-F(\Semit{\bar\nu}{{\frac{\tau}{2}}}{{\frac{\tau}{2}}})-
F(\Semit{\bar\nu}{\tau}{{\frac{\tau}{2}}})\right)\\
&\leq& 2\tau\left(F(\bar\nu)-F(\Semit{\bar\nu}{\tau}{{\frac{\tau}{2}}})\right).
\end{eqnarray}
Notice that \eqref{dyadic2} corresponds to \eqref{dyadic1} with $t=\tau$;
by adding the inequalities analogous to \eqref{dyadic2} between consecutive
times $m\tau$, $(m+1)\tau$, for $m=0,\ldots,N-1$, we obtain
\begin{equation}\label{dyadic3}
W_2^2(\Semit{\bar\mu}{N\tau}{\tau},\Semit{\bar\nu}{N\tau}{{\frac{\tau}{2}}})-
W_2^2(\bar\mu,\bar\nu)\leq 2
\tau\left(F(\bar\nu)-F(\Semit{\bar\nu}{N\tau}{{\frac{\tau}{2}}}\right),
\end{equation}
that yields \eqref{dyadic1} because $F$ is nonnegative.
Now, from \eqref{dyadic1} with $\bar\mu=\bar\nu$ we get
$$
W_2(\Semit{\bar\mu}{t}{{\frac{\tau}{2^m}}},
\Semit{\bar\mu}{t}{\frac{\tau}{2^{m+1}}})
\leq 2^{-m/2}\sqrt{2\tau F(\bar\mu)}
$$
for all $t$ that are integer multiples of $\tau/2^m$,
so that
\begin{equation}\label{beni4}
W_2(\Semit{\bar\mu}{t}{\frac{\tau}{2^m}},
\Semit{\bar\mu}{t}{\frac{\tau}{2^n}})
\leq \sum_{i=m}^{n-1} 2^{-i/2}\sqrt{2\tau F(\bar\mu)}
\end{equation}
for all $n>m\geq j$ and all $t$ that is an integer
multiple of $\tau/2^j$. For any such $t$ (and therefore on a dense set
of times) the sequence $(\Semit{\bar\mu}{t}{\frac{\tau}{2^n}})$ has the
Cauchy property and converges in $\ProbabilitiesTwo{H}$ to some limit,
that we shall denote by $\Semii{\bar\mu}{t}{\tau}$.

Using the discrete $C^{1/2}$
estimate \eqref{eeeuler1} we obtain convergence for all times, as
well as the uniform H\"older continuity \eqref{sat1}
of $t\mapsto\Semii{\bar\mu}{t}{\tau}$.

We prove now that $(\Semii{\bar\mu}{t}{\tau})_{t\geq 0}$
is a gradient flow starting from $\bar\mu$.
Indeed, we can read \eqref{discrevibis} as follows:
$$
\frac{d}{dt} \frac{1}{2} W_2^2(\Semit{\bar\mu}{t}{\tau},\nu)\leq
\tau\sum_{i=1}^\infty
[F(\nu)-F(\Semit{\bar\mu}{i \tau}{\tau})]\delta_{\frac{i}{\tau}}
$$
for all $\nu\in D(F)$, in the sense of distributions.
Passing to the limit as $n\to\infty$ in the previous inequality
with $\tau$ replaced by $\tau/2^n$, the lower semicontinuity
of $F$ gives
$$
\frac{d}{dt} \frac{1}{2} W_2^2(\Semii{\bar\mu}{t}{\tau},\nu)\leq
[F(\nu)-F(\Semii{\bar\mu}{t}{\tau})]\qquad\forall \nu\in D(F)
$$
in the sense of distributions. This proves that
$\Semii{\bar\mu}{t}{\tau}$ is a gradient flow starting from $\bar\mu$, and
since we proved that gradient flows are uniquely determined by the initial
condition, from now on we shall denote $\Semi{\bar\mu}{t}=\Semii{\bar\mu}{t}{\tau}$.

\smallskip\noindent
\emph{Proof of (i).} Passing to the limit as $n\to\infty$ in \eqref{beni4},
with $m=j=0$, we obtain that $W_2(\Semi{\bar\mu}{t},\Semit{\bar\mu}{t}{\tau})$
can be estimated with $2(\sqrt{2}+1)\sqrt{\tau F(\bar\mu)}$ when $t/\tau$
is an integer. From \eqref{eeeuler1}, \eqref{sat1}
and the triangle inequality we obtain (i).

\smallskip\noindent
\emph{Proof of (ii) when $\bar\mu\in D(F)$.}
It suffices to pass to the limit as $\tau\downarrow 0$
in \eqref{dyadic1}.

\medskip\noindent
\emph{Proof of (iii) when $\bar\mu\in D(F)$.} By adding the inequalities
\begin{eqnarray*}
W_2^2(\Semit{\bar\mu}{(i+1)\tau}{\tau},\nu)-W_2^2(\Semit{\bar\mu}{i\tau}{\tau},\nu)
&\leq& 2\tau [F(\nu)-F(\Semit{\bar\mu}{(i+1)\tau}{\tau})]\\&\leq&
2\tau [F(\nu)-F(\Semit{\bar\mu}{N\tau}{\tau})]
\end{eqnarray*}
for $i=0,\ldots,N-1$ we get
$$
W_2^2(\Semit{\bar\mu}{N\tau}{\tau},\nu)-W_2^2(\bar\mu,\nu) \leq
2N\tau [F(\nu)-F(\Semit{\bar\mu}{N\tau}{\tau})].
$$
Replacing now $\tau$ by $\tau/2^m$ in this inequality, and defining
$N$ as the integer part of $2^m t/\tau$ (so that $N\tau/2^m\to t$),
we can let $m\to\infty$ to obtain (iii),
neglecting the term $W_2^2(\Semit{\bar\mu}{N\tau}{\tau},\nu)$.

\smallskip\noindent
In order to prove (ii) and (iii) when $\bar\mu\in\overline{D(F)}$
we use a density argument. Indeed, let
$\bar\mu_n\in D(F)$ be converging to $\bar\mu\in\overline{D(F)}$
in $\ProbabilitiesTwo{H}$:
by (ii) we obtain that $\Semi{\bar\mu_n}{t}$ is a Cauchy sequence for
all $t\geq 0$, and therefore converges to some limit, that we shall denote
by $\Semi{\bar\mu}{t}$. It is not difficult to prove by approximation that
$\Semi{\bar\mu}{t}$ is a gradient flow, and it remains to show that it
starts from $\bar\mu$. We have indeed
$W_2(\Semi{\bar\mu}{t},\Semi{\bar\mu_n}{t})\leq W_2(\bar\mu_n,\bar\mu)$, so
that
$$
\limsup_{t\downarrow 0}W_2(\Semi{\bar\mu}{t},\bar\mu)\leq
2W_2(\bar\mu_n,\bar\mu)+\limsup_{t\downarrow 0} W_2(\Semi{\bar\mu_n}{t},\bar\mu_n)=
2W_2(\bar\mu_n,\bar\mu).
$$
Letting $n\to\infty$ we obtain that $\Semi{\bar\mu}{t}\to\bar\mu$ as $t\downarrow 0$.
\end{proof}

\section{$\Gamma$-convergence and stability properties}

In this section we consider a sequence $(\gamma_n)$
of log-concave probability measures on $H$ weakly converging to
$\gamma$ and a sequence of Hilbertian norms
on $\Hn =H^0(\gamma_n)$ satisfying Assumption~\ref{h_n}.
We are going to prove that the gradient flows
associated with $\RelativeEntropy{\cdot}{\gamma_n}$ with respect to
$W_{2,H^n}$ converge to the gradient flow
associated with $\RelativeEntropy{\cdot}{\gamma}$ with respect to $W_2$,
where the notation $W_{2,H^n}$ has been introduced in \eqref{www}.

This result is natural in view of Theorem \ref{tgflow}, since
the discrete approximating flow $\Semitn{\bar\mu^n}{\cdot}{\tau}$
of $\RelativeEntropy{\cdot}{\gamma_n}$ are defined only
in terms of $\gamma_n$ and $W_{2,H^n}$.
However, the same result is much less obvious
in view of the connection with the Fokker-Planck
equation (\ref{fk}) and the associated stochastic process
$(X_t)_{t\geq 0}$: see Sections \ref{dirichlet} and Section~\ref{process}.

The main result of this section is the following:
\begin{theorem}[Stability of gradient flows]\label{stabflows}
  Suppose that $(\gamma_n)\subset\Probabilities{H}$ is a sequence of log-concave
  probability measures converging weakly to $\gamma\in\Probabilities H$ and
  that Assumption~\ref{h_n} holds. Let $\bar\mu^n\in\ProbabilitiesTwo{A_n}$
  and let $(\mu^n_t)_{t\geq 0}$ be
  the gradient flows associated with $\RelativeEntropy{\cdot}{\gamma_n}$
  in $\ProbabilitiesTwo{A_n}$ with respect to $W_{2,H^n}$.\\
  If $\bar\mu_n$ converge to $\bar\mu\in\ProbabilitiesTwo{A}$ in $\ProbabilitiesTwo H$
  then $\mu^n_t\to\mu_t$ in $\ProbabilitiesTwo{H}$ for every $t\in [0,+\infty)$,
  where $(\mu_t)_{t\geq 0}$ is
  the gradient flow associated with $\RelativeEntropy{\cdot}{\gamma}$
  in $\ProbabilitiesTwo{A}$ with respect to $W_2$.
\end{theorem}
The crucial property in the proof of this stability
result is the $\Gamma$-convergence of the functionals
$\RelativeEntropy{\cdot}{\gamma_n}$ to $\RelativeEntropy{\cdot}{\gamma}$.
The concept of $\Gamma$-convergence is due to De Giorgi and
is a classical tool of Calculus of Variations.
\begin{lemma}[Convergence of entropy functionals] \label{lprod}
If $\gamma_n\in\Probabilities{H}$ converge
weakly to $\gamma\in\Probabilities{H}$
then $\RelativeEntropy{\cdot}{\gamma_n}:\ProbabilitiesTwo{H}\to[0,+\infty]$
\ $\Gamma$-converge to
$\RelativeEntropy{\cdot}{\gamma}:\ProbabilitiesTwo{H}\to[0,+\infty]$,
i.e.
\begin{itemize}
\item[(i)] for any sequence $(\mu_n)\subset\ProbabilitiesTwo{H}$ converging
\underline{weakly} to $\mu\in\ProbabilitiesTwo{H}$, we have
\begin{equation}\label{liminfc}
\liminf_{n\to\infty} \, \RelativeEntropy{\mu_n}{\gamma_n} \, \geq \,
\RelativeEntropy{\mu}{\gamma};
\end{equation}
\item[(ii)] for any $\mu\in\ProbabilitiesTwo{H}$
there exist $\mu_n\in\ProbabilitiesTwo{H}$
converging to $\mu$ \underline{in $\ProbabilitiesTwo{H}$}
such that
\begin{equation}\label{limsupc}
\limsup_{n\to\infty} \, \RelativeEntropy{\mu_n}{\gamma_n} \, \leq
\, \RelativeEntropy{\mu}{\gamma}.
\end{equation}
\end{itemize}
\end{lemma}
\begin{proof}
The ``liminf'' inequality (i) in the definition of $\Gamma$-convergence
follows directly from the duality formula \eqref{duality}: if $\mu_n\to\mu$
weakly, for all bounded continuous $S:H\to\R$ we have
$$
\int_H S\,d\mu-\int_H (e^S-1)\,d\gamma=\lim_{n\to\infty}
\left[\int_H S\,d\mu_n-\int_H (e^S-1)\,d\gamma_n\right]
\leq\liminf_{n\to\infty} \RelativeEntropy{\mu_n}{\gamma_n}.
$$
Taking the supremum in the left hand side the $\liminf$ inequality is achieved.

In order to show the $\limsup$ inequality we first notice that, by diagonal
arguments, we need only to show it for a dense subset
$\mathscr R\subset D(\RelativeEntropy{\cdot}{\gamma})$;
here density should be understood in the sense that for any
$\nu\in D(\RelativeEntropy{\cdot}{\gamma})$
there exist $\nu_n\in \mathscr R$ converging to $\nu$ in $\ProbabilitiesTwo{H}$ with
$\RelativeEntropy{\nu_n}{\gamma}\to\RelativeEntropy{\nu}{\gamma}$.
Let us check that
$$\mathscr R:=\left\{ e^{-\eps\|\cdot\|_H^2}f\gamma\in \ProbabilitiesTwo H:\ f\in C_b(H),\,\,f\geq 0,\ \eps>0\right\}$$
has these properties: indeed, in this case, given $\mu=g\gamma\in \mathscr R$ with
$g(x)=e^{-\eps\|x\|_H^2}f(x)$, we can simply
take $\mu_n=Z_n^{-1} g\gamma_n$, with $Z_n:=\int_H g\,d\gamma_n$,
to achieve the $\limsup$ inequality. The ``density in energy''
of $\mathscr R$ in the sense described above can be achieved as follows: first, using
the density of $C_b(H)$ in $L^1(\gamma)$ and the dominated convergence theorem,
we see that any $\mu=\rho\gamma\in\ProbabilitiesTwo{H}$ with
$\rho\in L^\infty(\gamma)$ can be approximated by elements of $\mathscr R$.
A truncation argument then gives that any
$\mu\in D(\RelativeEntropy{\cdot}{\gamma})$ can be approximated.
\end{proof}

In order to clarify the structure of the proof of Theorem~\ref{stabflows},
it is useful to introduce the following concept:
we say that $\mu_n\in\Probabilities{H_n}$ converge with
moments to $\mu\in\Probabilities{H}$ if $\mu_n\to\mu$ weakly in
$\Probabilities{H}$ and $\int_{H_n}\|x\|^2_{H_n}\,d\mu_n\to
\int_H\|x\|_H^2\,d\mu$. Notice that for any open set $A\subset H$ we can
use \eqref{fatext1} to obtain
\begin{equation}\label{fatext2}
\liminf_{n\to\infty}\int_A\|x\|_{H_n}^2\,d\mu_n=
\liminf_{n\to\infty}\int_A\|\pi_n(x)\|_{H_n}^2\,d\mu_n\geq
\int_A \|x\|_H^2\,d\mu
\end{equation}
whenever $\mu_n\to\mu$ weakly in $\Probabilities{H}$. Therefore, in the
proof of convergence with moments, only the $\limsup$ needs to be proved.

\begin{lemma} \label{lprodd}
Convergence with moments is equivalent to convergence in $\ProbabilitiesTwo{H}$.
\end{lemma}
\begin{proof} If $\mu_n\to\mu$ with moments, \eqref{fatext2} with $A=\{\|x\|_H<R\}$ gives
\begin{eqnarray*}
\lim_{R\to\infty}\limsup_{n\to\infty}\int_{\{\|x\|_H\geq R\}}\|x\|_H^2\,d\mu_n
&\leq&\kappa^2
\lim_{R\to\infty}\limsup_{n\to\infty}\int_{\{\|x\|_H\geq R\}}\|\pi_n(x)\|_{H_n}^2\,d\mu_n\\
&\leq&\kappa^2\lim_{R\to\infty}\int_{\{\|x\|_H\geq R\}}\|x\|_H^2\,d\mu=0.
\end{eqnarray*}
We obtain the convergence in $\ProbabilitiesTwo{H}$ from \eqref{cami2}.
Conversely, if $\mu_n\to\mu$ weakly, \eqref{fatext1} gives
\begin{equation}\label{trunca}
\lim_{n\to\infty}
\int_{H_n}\|\pi_n(x)\|_{H_n}^2\wedge R^2\,d\mu_n=
\int_H\|x\|_H^2\wedge R^2\,d\mu
\quad\qquad\forall R>0.
\end{equation}
If $\mu_n\to\mu$ in $\ProbabilitiesTwo{H}$, we can use \eqref{cami2}
and \eqref{bastakappa} to obtain
$\limsup_n\int_{\{\|x\|_{H_n}\geq R\}}\|x\|_{H_n}^2\,d\mu_n\to 0$
as $R\to\infty$, and if we combine this information with \eqref{trunca}
we obtain the convergence with moments.
\end{proof}

\begin{lemma} \label{lproddd} Assume that
$\mu_n,\,\nu_n\in\Probabilities{H_n}$, that
$\Sigma_n\in \Gamma(\mu_n,\nu_n)$ is
converging to $\Sigma\in\Gamma(\mu,\nu)$ weakly
and that $\mu_n\to\mu$ with moments, while
$\int_{H_n}\|y\|_{H_n}^2\,d\nu_n$ is bounded.
Then
$$
\lim_{n\to\infty}\int_{H_n\times H_n}\la x,y\ra_{H_n}\,d\Sigma_n=
\int_{H\times H}\la x,y\ra_H\,d\Sigma.
$$
\end{lemma}
\begin{proof}
We prove the $\liminf$ inequality only, the proof of the other one
being similar. With the notation of \eqref{sig}, we have
$\la \pi_n(x),\pi_n(y)\ra_{H_n}\to\la x,y\ra_H$ as $n\to\infty$
for all $x,\,y\in H$. For all $\varepsilon>0$ the functions
$$
\frac{1}{2\varepsilon}\|\pi_n(x)\|^2_{H_n}+
\frac{\varepsilon}{2} \|\pi_n(y)\|_{H_n}^2+2\la \pi_n(x),\pi_n(y)\ra_{H_n}
$$
are nonnegative, and these functions are equi-continuous in
$H\times H$ by \eqref{bastakappa}. Therefore \eqref{fatext1},
thanks to the convergence assumption on $\mu_n$, gives
$$
\liminf_{n\to\infty}\int_{H\times H}
\frac{\varepsilon}{2} \|\pi_n(y)\|_{H_n}^2+2\la\pi_n(x),\pi_n(y)\ra_{H_n}\,d\Sigma_n
\geq\int_{H\times H}
\frac{\varepsilon}{2} \|y\|_H^2+2\la x,y\ra_{H}\,d\Sigma.
$$
Using the boundedness assumption on $(\nu_n)$ we can obtain
the $\liminf$ inequality letting $\varepsilon\downarrow 0$.
\end{proof}

In the proof of Theorem \ref{stabflows} we need some continuity/lower
semicontinuity properties of the Wasserstein distance.

 \begin{lemma}\label{gW}
 Let $\mu_n,\,\nu_n\in\ProbabilitiesTwo{H_n}$ be such that
 $\mu_n\to\mu\in\ProbabilitiesTwo{H}$,
 $\nu_n\to\nu\in\ProbabilitiesTwo{H}$ weakly in $\Probabilities{H}$. Then:
 \begin{itemize}
 \item[(i)] $W_2(\mu,\nu)\leq\liminf\limits_{n\to\infty} W_{2,\Hn}(\mu_n,\nu_n)$;
 \item[(ii)] if $\mu_n\to\mu$ and $\nu_n\to\nu$ in $\ProbabilitiesTwo{H}$, then
  $W_{2,H_n}(\mu_n,\nu_n)\to W_2(\mu,\nu)$.
 \end{itemize}
\end{lemma}
 \begin{proof} (i) Without loss of generality, we can assume that the
 $\liminf$ is a limit. Let $\Sigma_n\in\Gamma_{H_n,o}(\mu_n,\nu_n)$.
 Notice that tightness of $(\mu_n)$ and $(\nu_n)$ in $H$ implies
 tightness of $(\Sigma_n)$ in $H\times H$. Let
 $\Sigma\in\Gamma(\mu,\nu)$ be a weak limit point of $(\Sigma_n)$,
 which obviously belongs to $\Gamma(\mu,\nu)$.
 Then, taking into account the equi-continuity in $H\times H$ of
 the maps $\|\pi_n(x-y)\|_{H_n}$, ensured by \eqref{bastakappa},
 by \eqref{fatext1} we get:
 \begin{eqnarray*}
  W_2^2(\mu,\nu)&\leq&
  \int_{H\times H}\|y-x\|_H^2\,d\Sigma\leq\liminf_{n\to\infty}
  \int_{H\times H}\|\pi_n(x-y)\|^2_{\Hn }\,d\Sigma_n\\
  &=&\liminf_{n\to\infty}\int_{H_n\times H_n}\|y-x\|_{\Hn}^2\,d\Sigma_n
  =\liminf_{n\to\infty} W_{2,\Hn}^2(\mu_n,\nu_n).
  \end{eqnarray*}
  (ii) We choose optimal couplings
  $\Sigma_n$ between $\mu_n$ and $\nu_n$, relative to $H_n$,
  and prove that any weak limit $\Sigma$
  (which exists, possibly passing to subsequences) is optimal.
  The same truncation argument used in Lemma~\ref{lprodd} to
  show that convergence in $\ProbabilitiesTwo{H}$ implies convergence with moments
  shows that
  $$
  \lim_{n\to\infty}
  \int_{\Hn \times \Hn }\|y-x\|_{\Hn }^2\,d\Sigma_n
  =\int_{H\times H}\|y-x\|_{H}^2\,d\Sigma.
  $$
  In order to prove the optimality of $\Sigma$
  we recall that $\Sigma\in\Gamma(\lambda,\nu)$ is
  an optimal coupling (relative to the cost $c(x,y)=\|x-y\|^2_H$)
  if and only if for any $\ell\in\N$, any
  $(x_i,y_i)_{i=1,\ldots,\ell}$ in the support of $\Sigma$ and
  any permutation $\sigma$ of $\{1,\ldots,\ell\}$ the following
  inequality holds:
  \begin{equation}\label{natural}
  \sum_{i=1}^\ell \|x_i-y_{\sigma(i)}\|^2_H  \, \geq \,
  \sum_{i=1}^\ell \|x_i-y_i\|^2_H,
  \end{equation}
  see for instance \cite[Theorem 6.1.4]{ags}.
  Since $\Sigma_n$ is optimal, a similar inequality holds
  with $\|\cdot\|_{\Hn }$ instead of $\|\cdot\|_H$ for
  all $(x_i^n,y_i^n)_{i=1,\ldots,\ell}$ in the support of $\Sigma_n$. Since
  $\Sigma_n$ converge to $\Sigma$ weakly, for any
  $(x_i,y_i)_{i=1,\ldots,\ell}$ in the support of $\Sigma$
  there exist $(x_i^n,y_i^n)_{i=1,\ldots,\ell}$ in the support of $\Sigma_n$
  with $(x_i^n,y_i^n)\to(x_i,y_i)$ in $H\times H$. Then
  \eqref{natural} follows taking limits as $n\to\infty$ and using
  the fact that $z_n\in H_n$ and $\|z_n-z\|_H\to 0$ implies
  $\|z_n\|_{H_n}\to\|z\|_H$.
 \end{proof}

We can now prove Theorem~\ref{stabflows}. With no loss of generality
we can assume (possibly making translations) that $A(\gamma_n)=H_n$.

\begin{proof}
Set $F_n(\cdot):=\RelativeEntropy{\cdot}{\gamma_n}$.
We consider the case when $F_n(\bar\mu^n)$ is bounded first. In this
case, property (i) in Theorem~\ref{tgflow} and \eqref{bastakappa}
ensure the uniform (in time, and with
respect to $n$) estimate $W_2^2(\mu^n_t,\Semitn{\bar\mu^n}{t}{\tau})\leq C\tau$.
Here $\Semitn{\bar\mu^n}{t}{\tau}$ is the discrete approximation \eqref{recu}
of the gradient flow, obtained
by the recursive minimization scheme \eqref{euler15}: i.e. we define recursively
$\mu^{n,0}_\tau:=\bar\mu^n$, $\mu^{n,k+1}_\tau$ is the unique minimizer of
\[
\ProbabilitiesTwo{H_n} \, \ni \, \nu \, \mapsto \,
F_n(\nu)+\frac{1}{2\tau}W_{2,H_n}^2(\nu,\mu^{n,k}_\tau),
\]
and we define $\Semitn{\bar\mu^n}{t}{\tau}:=\mu^{n,k+1}_\tau$ for all
$t\in (k\tau,(k+1)\tau]$. Therefore, taking also Lemma~\ref{lprodd}
into account, in this case it suffices to
show that, with $\tau>0$ \emph{fixed}, the convergence with
moments is preserved by the minimization scheme.
So, let us assume that $\mu_n$ converge to $\mu$ with moments
and $F_n(\mu_n)$ is bounded; we consider
the minimizers $\nu_n$ of the problems
\[
\ProbabilitiesTwo{H_n} \, \ni \, \sigma \, \mapsto \,
F_n(\sigma)+\frac{1}{2\tau}W_{2,H_n}^2(\sigma,\mu_n),
\]
and show that they converge with moments to the minimizer $\nu$
of the problem
\begin{equation}\label{minimi}
\ProbabilitiesTwo{H} \, \ni \, \sigma \, \mapsto \,
F(\sigma)+\frac{1}{2\tau}W_{2}^2(\sigma,\mu).
\end{equation}
Notice first we can use $\lambda=\mu_n$ in the inequality
\begin{equation}\label{mino}
F_n(\nu_n)+\frac{1}{2\tau}W_{2,H_n}^2(\nu_n,\mu_n)\leq
F_n(\lambda)+\frac{1}{2\tau}W_{2,H_n}^2(\lambda,\mu_n)
\end{equation}
to obtain that both $\int_{H_n}\|y\|^2_{H_n}\,d\nu_n$
and $F_n(\nu_n)$ are bounded. Since $(\gamma_n)$ is tight
and $\RelativeEntropy{\nu_n}{\gamma_n}$ is bounded,
then $(\nu_n)$ is tight as well, by the entropy inequality \eqref{entrine}.
Therefore $(\nu_n)$ has limit points with respect to the weak convergence.
We will prove that any limit point is a minimizer of
\eqref{minimi}, so that it must be $\nu$. \\
Let $\nu'=\lim_k\nu_{n(k)}$ in the weak convergence,
let $\lambda\in\ProbabilitiesTwo{H}$ and let $\lambda_k$ be converging to
$\lambda$ in $\ProbabilitiesTwo{H}$, with $\limsup_k F_{n(k)}(\lambda_k)\leq F(\lambda)$,
whose existence is ensured by condition (ii) in the definition of $\Gamma$-convergence.
Setting $\lambda=\lambda_k$, $n=n(k)$ in \eqref{mino}, and using also condition (i) in the
definition of $\Gamma$-convergence to bound $F_{n(k)}(\nu_{n(k)})$ from below, we get
from (i) and (ii) of Lemma~\ref{gW}
\begin{eqnarray}\label{minobis}
F(\nu')+\frac{1}{2\tau}W_2^2(\nu',\mu)&\leq&
\limsup_{k\to\infty} \left[F_{n(k)}(\nu_{n(k)})+
\frac{1}{2\tau}W^2_{2}(\nu_{n(k)},\mu_{n(k)})\right]
\nonumber\\
&\leq&
\limsup_{k\to\infty} \left[F_{n(k)}(\lambda_k)+
\frac{1}{2\tau}W_{2}^2(\lambda_k,\mu_{n(k)})\right]\nonumber \\
&\leq&
F(\lambda)+\frac{1}{2\tau}W_2^2(\lambda,\nu).
\end{eqnarray}
As $\lambda$ is arbitrary, this proves that $\nu'$ is a minimizer, therefore $\nu'=\nu$.

Now, setting $\lambda=\mu$ in \eqref{minobis}, we obtain that all
inequalities must be equalities, so that $\lim_kW_{2}^2(\mu_{n(k)},\nu_{n(k)})=W_2^2(\mu,\nu)$.
Indeed, if $\limsup_k(a_k+b_k)\leq a+b$, $\liminf_k a_k\geq a$
and $\liminf_k b_k\geq b$, then $\lim_k a_k=a$ and $\lim_k b_k=b$.

We shall denote in the sequel by $\Sigma_n$ optimal couplings between $\mu_n$ and $\nu_n$.
Let $\Sigma\in\Gamma(\mu,\nu)$ a limit point in the weak convergence
of $\Sigma_n$, and assume just for notational simplicity that the whole
sequence $\Sigma_n$ weakly converges to $\Sigma$. By \eqref{fatext1}
we get
$$
\int_{H\times H}\|x-y\|_H^2\,d\Sigma \, \leq \,
\liminf_{n\to\infty}\int_{H\times H}\|x-y\|^2_{H_n}\,d\Sigma_n= W_2^2(\mu,\nu),
$$
therefore $\Sigma\in\Gamma_o(\mu,\nu)$. We can now apply Lemma~\ref{lproddd}
to obtain that
$\int_{H_n\times H_n}\la x,y\ra_{H_n}\,d\Sigma_n\to\int_{H\times H}\la x,y\ra\,d\Sigma$;
from the identity
$$
W_2^2(\mu,\nu)=\int_H\|x\|^2_{H}\,d\mu+
\int_H\|y\|^2_H\,d\nu-2\int_{H\times H}\la x,y\ra_H\,d\Sigma,
$$
and from the analogous one with the Hilbert spaces $H_n$ we obtain that
$\nu_n$ converge with moments to $\nu$.

In the general case when $F_n(\bar\mu_n)$ is not bounded we can find, for any
$\varepsilon>0$, $\bar\nu\in D(F)$ with $W_2(\bar\mu,\bar\nu)<\varepsilon$.
By the definition of $\Gamma$-convergence we can also find $\bar\nu^n$ converging
to $\bar\nu$ in $\ProbabilitiesTwo{H}$ with $\limsup_n F_n(\bar\nu_n)\leq F(\bar\nu)$. For
$n$ large enough we still have $W_{2,H_n}(\bar\mu^n,\bar\nu^n)<\varepsilon$, so that
the contracting property of gradient flows (see Theorem~\ref{tgflow} (ii)) gives
\[
\begin{split}
& W_2(\Semi{\bar\mu^n}{t},\Semi{\bar\nu^n}{t})+
W_2(\Semi{\bar\mu}{t},\Semi{\bar\nu}{t}) \\ &  \leq \, \kappa
W_{2,H_n}(\Semi{\bar\mu^n}{t},\Semi{\bar\nu^n}{t})+
W_2(\Semi{\bar\mu}{t},\Semi{\bar\nu}{t}) <(\kappa+1)\varepsilon,
\qquad\forall t\geq 0.
\end{split}
\]
By applying the local uniform convergence property to $\bar\nu^n$ we get
$$
\limsup_{n\to\infty}\sup_{t\in [0,T]}
W_2(\Semi{\bar\mu^n}{t},\Semi{\bar\mu}{t})\leq (\kappa+1)\varepsilon
\qquad\forall T>0.
$$
\end{proof}

\section{Wasserstein semigroup and Dirichlet forms}
\label{dirichlet}

In this section we establish a general link between the Wassertein
semigroups and the semigroups arising from natural ``gradient''
Dirichlet forms, extending Proposition \ref{prop:Fkpre} to the
general case of a log-concave measure $\gamma$ in $H$. We denote by
$K$ the support of $\gamma$ (a closed convex set, coinciding with
$\overline{\{V<+\infty\}}$ when $H=\R^k$ and
$\gamma=\exp(-V)\Leb{k}$) and, without a real loss of generality, we
consider the case when
\begin{equation}\label{de1}
A(\gamma) \ = \ H^0(\gamma) \ = \ H.
\end{equation}
We consider, recalling \eqref{defah}, the bilinear form
\begin{equation}\label{cami4}
\cE_\gamma(u,v) \, := \, \int_H \la\nabla u,\nabla v\ra \, d\gamma,
\qquad u,\,v\in C^1_b(H).
\end{equation}
Accordingly, we define the induced scalar product and norm on $C^1_b(H)$:
\begin{equation}\label{scpr}
\cE_{\gamma,1}(u,v) \, := \, \int_H u \, v \, d\gamma \, + \,
\cE_\gamma(u,v), \qquad \|u\|_{\cE_{\gamma,1}} \, := \,
\sqrt{\cE_{\gamma,1}(u,u) }.
\end{equation}
We start proving that $\cE_\gamma$ is closable.
We recall that closability means the following:
for all sequences $(u_n)\subset C^1_b(H)$ which are Cauchy
with respect to $\|\cdot\|_{\cE_{\gamma,1}}$ and such that $u_n\to 0$
in $L^2(\gamma)$, we have $\|u_n\|_{\cE_{\gamma,1}}\to 0$. This is equivalent
to saying that the operator $\nabla:C^1_b(H)\mapsto L^2(\gamma;H)$ is
closable in $L^2(\gamma)$.
\begin{lemma}[Closability]\label{closability}
The bilinear form $(\cE_\gamma,C^1_b(H))$ is closable in $L^2(\gamma)$.
\end{lemma}
\begin{proof} Let us denote by ${\rm Cyl}(H)$ the subspace of $C^1_b(H)$ made by
cylindrical functions; by a simple density argument we can assume
that the sequence $(u_n)$ is contained in ${\rm Cyl}(H)$. We claim
that closability follows by the lower semicontinuity of $v\mapsto
\cE_\gamma(v,v)$ on ${\rm Cyl}(H)$: indeed, if this lower
semicontinuity property holds, we can pass to the limit as
$m\to\infty$ in the inequality
$\cE_\gamma(u_n-u_m,u_n-u_m)<\varepsilon$, for $n,\,m\geq
n(\varepsilon)$, to obtain $\cE_\gamma(u_n,u_n)<\varepsilon$ for
$n\geq n(\varepsilon)$, i.e. $\|u_n\|_{\cE_{\gamma,1}}\to 0$.

So, let $(v_n)\subset {\rm Cyl}(H)$ be converging in $L^2(\gamma)$ to $v\in {\rm Cyl}(H)$
and let us prove that the inequality $\liminf_n\cE_\gamma(v_n,v_n)\geq\cE_\gamma(v,v)$
holds.

We show first that we can assume with no loss of generality that
$\gamma\in\ProbabilitiesTwo{H}$, so that
$f\gamma\in\ProbabilitiesTwo{H}$ for all bounded Borel functions
$f$. Indeed, we can approximate $\gamma$ by the log-concave measures
$\gamma_\varepsilon:=\exp(-\varepsilon\|x\|^2)\gamma/Z_\varepsilon\in\ProbabilitiesTwo{H}$,
where $Z_\varepsilon\uparrow 1$ are normalization constants, and use
the fact that $\gamma_\varepsilon\leq\gamma/Z_\varepsilon$ and
$Z_\varepsilon\gamma_\varepsilon\uparrow\gamma$ to obtain the lower
semicontinuity of $\cE_\gamma(v,v)$ from the lower semicontinuity of
all $\cE_{\gamma_\varepsilon}(v,v)$. The log-concavity of
$\gamma_\varepsilon$ can be obtained by approximation: if $(\ee_i)$
is an orthonormal system in $H$, then all measures
$$
\gamma_{\varepsilon,N}:=\frac{1}{Z_{\varepsilon,N}}\exp(-\sum_{i=1}^N\langle x,\ee_i\rangle^2)\gamma
$$
are log-concave because their projections on any finite-dimensional subspace
$L\supset (\ee_1,\ldots,\ee_N)$ have the form
$Z_{\varepsilon,M}^{-1}\exp(-\varepsilon\sum_1^N\langle x,\ee_i\rangle^2-V)$,
where $\exp(-V)$ is the density of $(\pi_L)_\#\gamma$. Therefore Proposition~\ref{charalog}
can be applied.

We can assume, possibly adding and multiplying by constants,
that $m:=\inf v>0$ and $\int v^2\,d\gamma=1$. By a simple truncation argument we can also
assume that $\inf v_n\geq m/2$, $\sup v_n\leq\sup v+1$, and set $w_n=v_n/\|v_n\|_2$; obviously
$w_n\to v$ in $L^2(\gamma)$ and, as a consequence, $w_n^2\gamma\to v^2\gamma$ weakly.
By Lemma~\ref{slopeentro} we get
$$
\RelativeEntropy{\mu}{\gamma}\geq\RelativeEntropy{w_n^2\gamma}{\gamma}-
2\sqrt{\cE_\gamma(w_n,w_n)} \,
W_2(\mu,w_n^2\gamma)\qquad\forall\mu\in\ProbabilitiesTwo{H}.
$$
The uniform upper bound on $w_n$ ensures, taking \eqref{cami2} into
account, that $w_n^2\gamma\to v^2\gamma$ in $\ProbabilitiesTwo{H}$.
Passing to the limit as $n\to\infty$, the lower semicontinuity of
the relative Entropy gives
$$
\RelativeEntropy{\mu}{\gamma}\geq\RelativeEntropy{v^2\gamma}{\gamma}-
2\liminf_{n\to\infty}\sqrt{\cE_\gamma(w_n,w_n)} \,
W_2(\mu,v^2\gamma)\qquad\forall\mu\in\ProbabilitiesTwo{H}.
$$
By applying Lemma~\ref{slopeentro} again we get $\liminf_n\cE_\gamma(w_n,w_n)\geq\cE_\gamma(v,v)$,
and from the definition of $w_n$ we see that the same inequality holds if we replace $w_n$
with $v_n$.
\end{proof}

Being $\cE_\gamma$ closable, we shall denote by $D(\cE_\gamma)$ its domain
(i.e. the closure of $C^1_b(H)$ with respect to the norm $\Vert\cdot\Vert_{\cE_{\gamma,1}}$),
which obviously can be identified with a subset of $L^2(\gamma)$,
and keep the notation $\cE_\gamma$ for the extension of $\cE_\gamma$ to
$D(\cE_\gamma)\times D(\cE_\gamma)$.
In the next lemma we show that $D(\cE_\gamma)$ contains ${\rm Lip}_b(K)$ and some
useful representation formulas for the extension.

Recall that a \emph{finite signed measure} is a $\R$-valued set function
defined on Borel sets that can be written as the difference of two positive finite
measures; by Hahn decomposition, any such measure $\mu$ can be uniquely written as
$\mu=\mu^+-\mu^-$, with $\mu^\pm$ nonnegative and $\mu^+\perp\mu^-$. The
\emph{total variation} $|\mu|$ is the finite measure defined by $\mu^++\mu^-$.

\begin{lemma}[$\cE_\gamma$ is a Dirichlet form]\label{lipb}
$\cE_\gamma$ is a Dirichlet form, ${\rm Lip}_b(K)\subset D(\cE_\gamma)$  and
\begin{equation}\label{Markov}
\sqrt{\cE_\gamma(u,u)}\leq [u]_{\rm Lip(K)}\qquad\forall u\in {\rm Lip}_b(K).
\end{equation}
Moreover, the following properties hold:
\begin{itemize}
\item[(i)] if $H$ is finite-dimensional, $h\in H$ and $\ell_h(x)=\langle h,x\rangle$,
there exists a finite signed measure $\Sigma_h$ in $H$ supported on $K$ such that
\begin{equation}\label{ibpr}
\cE_\gamma(u,\ell_h) = \int_H u \, d\Sigma_h \qquad\forall u\in{\rm Lip}_b(K);
\end{equation}
\item[(ii)] if $\pi:H\to L$ is a finite-dimensional orthogonal projection, then
\begin{equation}\label{iprbis}
\cE_\gamma(u\circ\pi,v\circ\pi)=\cE_{\pi_\#\gamma}(u,v)
\qquad\forall u,\,v\in {\rm Lip}_b(L).
\end{equation}
\end{itemize}
\end{lemma}
\begin{proof}
Let $u\in{\rm Lip}_b(K)$ and let $\tilde u$ be a bounded Lipschitz extension
of $u$ to the whole of $H$. Combining finite-dimensional approximation and smoothing,
we can easily find a sequence $(u_n)\subset C^1_b(H)$
converging to $\tilde u$ pointwise and with $[u_n]_{{\rm Lip}(H)}$ bounded. It follows
that $u_n\to\tilde u$ in $L^2(\gamma)$ and, possibly extracting a subsequence,
$\nabla u_n\to U$ weakly in $L^2(\gamma;H)$. Then, a sequence $(g_n)$ of
convex combinations of $u_n$ still converges to $\tilde u$ in $L^2(\gamma;H)$ and
is Cauchy with respect to $\Vert\cdot\Vert_{\cE_{\gamma,1}}$. It follows that
$u\in D(\cE_\gamma)$.
A similar argument proves \eqref{Markov} and the fact that $\cE_\gamma(\phi(u),\phi(u))$
is less than $\cE_\gamma(u,u)$ whenever $u\in D(\cE_\gamma)$ and $\phi:\R\to\R$ is $1$-Lipschitz.
This last property shows that $\cE_\gamma$ is a Dirichlet form.

Now, let $H=\R^k$, so that $\gamma=\exp(-V)\Leb{k}$, and let us prove (i).
By the closability of $\cE_\gamma$, we need only to prove that
\begin{equation}\label{ibpri}
\int_{\R^k}\langle\nabla u,h\rangle\,d\gamma=\int_{\R^k}u\,d\Sigma_h
\qquad\forall u\in C^1_b(\R^k)
\end{equation}
for some finite signed measure $\Sigma_h$. The existence of such a measure $\Sigma_h$
(obvious in the case when $\nabla V$ is Lipschitz, as an integration
by parts gives $\Sigma_h=-\langle\nabla V,h\rangle\gamma$) is ensured
by Proposition~\ref{pzamb1}.

Finally, notice that \eqref{iprbis} trivially holds by the definitions of
$\cE_\gamma$ and $\cE_{\pi_\#\gamma}$ when $u\in C^1_b(L)$, because
$u\circ\pi\in C^1_b(H)$. By approximation the equality extends to the
case $u,\,v\in {\rm Lip}_b(L)$.
\end{proof}

By the previous lemma, there exists a unique contraction semigroup
$P_t$ in $L^2(\gamma)$ associated to $\cE_\gamma$. We are now going
to compare it with the Wasserstein semigroup $\Semi{\mu}{t}$ of
Theorem~\ref{tgflow}, and we shall denote in the sequel
$\nu_t^x:=\Semi{\delta_x}{t}$.

\begin{theorem}\label{nonsmooth}
The semigroup $P_t$ is regularizing from $L^\infty(\gamma)$ to $C_b(K)$,
and the identity
\begin{equation}\label{rdi1}
P_t f(x):=\int_H f\,d\nu_t^x,\qquad t>0,\,\,f\in
L^\infty(\gamma),\,\, x\in K
\end{equation}
provides a continuous version of $P_t f$. In addition, $P_t$ acts on
${\rm Lip}_b(K)$:
\begin{equation}\label{felle}
[P_t f]_{{\rm Lip}(K)}\leq [f]_{{\rm Lip}(K)}\qquad \ t\geq 0,
\quad\forall f\in {\rm Lip}_b(K).
\end{equation}
Moreover, for any $\mu\in\ProbabilitiesTwo{K}$, we have the identity
\begin{equation}\label{rdi12}
\Semi{\mu}{t} = \int \nu_t^x \, d\mu(x), \qquad \forall \, t\geq 0.
\end{equation}
\end{theorem}
\begin{proof} Assuming \eqref{rdi1}, let us first show why it provides
a continuous version of $P_t$: if $x_n\to x$, and we denote by
$\rho_n$ the densities of $\nu_{t}^{x_n}$ with respect to $\gamma$,
whose existence is ensured by the estimate
Theorem~\ref{tgflow}(iii), the contracting property of the semigroup
gives that $\rho_n\gamma\to\rho\gamma$ weakly, where $\rho$ is the
density of $\nu_t^x$ with respect to $\gamma$. On the other hand,
the same estimate shows that
$\RelativeEntropy{\rho_n\gamma}{\gamma}$ are uniformly bounded,
therefore $\rho_n$ are equi-integrable in $L^1(\gamma)$ and weakly
converge in $L^1(\gamma)$ to $\rho$. This proves that the right hand
side in \eqref{rdi1} is continuous. Finally, \eqref{felle} is a
direct consequence of \eqref{rdi1} and Theorem~\ref{stabflows}(iii):
indeed, choosing $\Sigma\in\Gamma_o(\nu_t^x,\nu_t^y)$, we get
\begin{eqnarray*}
|P_f(x)-P_tf(y)|&=&\left|\int_H f\,d\nu_t^x-\int_H
f\,d\nu_t^y\right|
=\left|\int_H (f(u)-f(v))\,d\Sigma(u,v)\right|\\
&\leq&[f]_{\rm Lip(K)}\int_H\|u-v\|\,d\Sigma(u,v)\leq
[f]_{\rm Lip(K)}W_2(\nu_t^x,\nu_t^y)\\
&\leq&[f]_{\rm Lip(K)}\|x-y\|.
\end{eqnarray*}

\smallskip
\noindent {\bf Step 1: the general finite-dimensional case
$H=\R^k$.} It suffices to show that the class of convex $V$'s for
which the equivalence \eqref{rdi1} holds for the probability measure
$\gamma=Z^{-1}\exp(-V)\, \Leb{k}$, is closed under monotone
convergence.

Indeed, if $V$ is smooth with $\nabla V$ and all its derivatives
bounded, then we know from Section~\ref{tuttolega} that the
Wasserstein semigroup coincides with the FP semigroup; therefore
from point 3 of Proposition \ref{prop:Fkpre} we obtain \eqref{rdi1},
and \eqref{felle} corresponds to \eqref{feller}.

By a convolution approximation, we extend the result to all convex
Lipschitz functions $V$ with $\int\exp(-V)\,dx<\infty$; indeed, if
$\rho_n$ is the density of $\cN(0,n^{-1} I)$ with respect to
$\Leb{k}$, where $I$ is the identity matrix in $\R^k$, then
$V_n:=V*\rho_n$ forms an increasing sequence by convexity of $V$.
Eventually we obtain all convex functions $V$ with
$\int\exp(-V)\,dx<\infty$ using the fact that they can be
represented as the supremum (see for instance \cite{br}) of
countably many affine functions $\ell_i$, and applying the
equivalence to $V_n:=\max_{1\leq i\leq n}\ell_i$ (notice that for
$n$ large enough $V_n$ has at least linear growth at infinity).

So, let us consider a log-concave probability measure
$\gamma=\exp(-V)\, \Leb{k}$ and a sequence $V_n\uparrow V$, with
$V_n$ real-valued and $V_1$ having at least a linear growth at
infinity, such that the statement of the theorem holds for all
measures $\gamma_n=Z_n^{-1}\exp(-V_n)\, \Leb{k}$; obviously the
normalization constants $Z_n$ converge to 1 and $\gamma_n\to\gamma$
weakly. Notice also that ${\rm supp\,}\gamma_n=\R^k$.

We will also use the fact that both sides in \eqref{rdi1} are continuous
with respect to $\gamma$-almost sure and dominated convergence, so we need
only to check the identity when $f\in {\rm Lip}_b(K)$.
We recall that, in general, the semigroup $P_t f$ is
related to the resolvent family $R_\lambda f$ by
\begin{equation}\label{respt}
R_\lambda f(x)= \int_0^\infty e^{-\lambda t}
P_tf(x)\, dt \qquad x\in K,\,\, f\in C_b(K).
\end{equation}
We define the bilinear form
\[
\cE^n(u,v) \, := \, \int_{\R^k} \la\nabla u,\nabla v\ra \, d\gamma_n,
\qquad u,\,v\in C^1_b(\R^k).
\]
Moreover, we denote by $R_\lambda^n$ the resolvent family of $\cE^n$,
again related to the semigroup $P^n_t$ on $C_b(\R^k)$ relative to $\cE^n$ by
$\int_0^\infty\exp(-\lambda t)P^n_t\,dt$. Using the representation
\eqref{rdi1} of $P^n_t$, by Theorem~\ref{stabflows} we know that,
for all $f\in C_b(\R^k)$, $R^n_\lambda f$ pointwise converge, on $K$,
to the function $F_\lambda f$ defined by
\[
F_\lambda f(x) := \int_0^\infty e^{-\lambda t} \, \int_H f\,\nu_t^x
\, dt \quad\qquad x\in K.
\]
We are going to show that $F_\lambda f$ coincides with $R_\lambda f$,
the resolvent family of $\cE_\gamma$, for all ${\rm Lip}_b(K)$, so that
\[
R_\lambda f(x)= \int_0^\infty e^{-\lambda t}
\int_{\R^k}f\,d\nu_t^x\, dt \qquad x\in K,\,\, f\in {\rm Lip}_b(K).
\]
Since, by the injectivity of the Laplace transform, \eqref{respt} uniquely determines
the semigroup $P_t$ on ${\rm Lip}_b(K)$, \eqref{rdi1} would be achieved.

So, let $f\in{\rm Lip}_b(K)$; possibly replacing $f$ by a Lipschitz extension
to the whole of $\R^k$ with the same Lipschitz constant, we can assume that
$f\in {\rm Lip}_b(\R^k)$ and $[f]_{{\rm Lip}(\R^k)}=[f]_{{\rm Lip}(K)}$
(indeed, neither $F_\lambda f$ nor $P_t f$ depend
on this extension). By applying \eqref{felle} to $\gamma_n$
one obtains that $\lambda [R_\lambda^n f]_{{\rm Lip}(\R^k)}\leq [f]_{{\rm Lip}(K)}$, hence
$F_\lambda f\in {\rm Lip}_b(K)$ and Lemma~\ref{lipb} gives
$F_\lambda f\in D(\cE_\gamma)$.
Now, in order to prove that $F_\lambda f$ coincides with $R_\lambda f$,
by a density argument it is enough to show that
\begin{equation}\label{st10}
\lambda \int_{\R^k}   F_\lambda f\, v \, d\gamma  + \cE_\gamma(F_\lambda f,v)
= \int_{\R^k} f \, v \, d\gamma \qquad \forall \, v\in C^2_c(\R^k).
\end{equation}
Our strategy is to pass to the limit as $n\to\infty$ in:
\begin{equation}\label{st1}
\int_{\R^k} fv\,d\gamma_n=\lambda \int_{\R^k} R_\lambda^nf  \, v \, d\gamma_n +
\cE^n(R_\lambda^nf,v)\qquad \forall \ v\in C^2_c(\R^k).
\end{equation}
Let $(\ee_1,\ldots,\ee_k)$ be the canonical basis of $\R^k$.
By applying the integration by parts formula \eqref{ibpr} with
$h=\ee_i$ and $u=\langle\nabla v,\ee_i\rangle R_\lambda^nf$, we get
\begin{equation}\label{rio1}
\int_{\R^k} f  v \, d\gamma_n =
\int_{\R^k} R^n_\lambda f (\lambda v-\Delta v) \,
d\gamma_n  + \sum_{i=1}^k \int_{\R^k} R^n_\lambda f \
\langle\nabla v,\ee_i\rangle  \, d\Sigma_{\see_i}^n
\end{equation}
where $\Sigma_{\see_i}^n$ are associated to the measure $\gamma_n$.
The crucial fact is now the following: we can apply Lemma \ref{easy} to
\[
\sigma_n  := (\lambda v- \Delta v)  \, d\gamma_n + \sum_{i=1}^k
\langle\nabla v,\ee_i\rangle  \, d\Sigma_{\see_i}^n,
\qquad \sigma_\infty  := (\lambda v- \Delta v)  \, d\gamma +  \sum_{i=1}^k
\langle\nabla v,\ee_i\rangle  \, d\Sigma_{\see_i},
\]
with $\varphi_n(x) := R^n_\lambda f(x)$, $\varphi_\infty(x)=F_\lambda f(x)$
and $\Sigma_{\see_i}$ associated to the measure $\gamma$. Indeed,
assumptions (i) and (ii) of the Lemma~\ref{easy} are guaranteed by
Proposition~\ref{pzamb2} in the Appendix, while (iii) and (iv) hold trivially.
Therefore, by \eqref{rio1} we have:
\[
\int_{\R^k} f v \, d\gamma  =  \lim_{n\to \infty}
\int_{\R^k} f v \, d\gamma_n = \int_{\R^k} F_\lambda f \left(\lambda v
- \Delta v\right) \, d\gamma + \sum_{i=1}^k \int_{\R^k} F_\lambda f \
\langle\nabla v,\ee_i\rangle  \, d\Sigma_{\see_i}.
\]
Again, by the integration by parts formula (\ref{ibpr}) shows that
the last expression is equal to the right hand side in (\ref{st10}).
This proves that $F_\lambda f=R_\lambda f$ on $K$ for all $f\in {\rm Lip}_b(K)$.

Notice now that \eqref{rdi12} holds for smooth $V$ by
\eqref{eq:dr1.5} and Proposition \ref{abs}. By approximation, using
the stability result of Theorem \ref{stabflows}, we obtain the
general case.

\smallskip
\noindent
{\bf Step 2: from the finite-dimensional to the infinite-dimensional case.}
We fix a complete orthonormal system
$\{\ee_i\}_{i\geq 1}$ in $H$ and we set $\Hn :={\rm span}\{\ee_1,\ldots,\ee_n\}$,
denoting as usual by $\pi_n:H\to\Hn$ the finite-dimensional projections. Setting
$\gamma_n=(\pi_n)_\#\gamma$, it is immediate to check that $\gamma_n$ is log-concave
in $H_n$ and that $H^0(\gamma_n)=H_n$ (if not, we would get that $H^0(\gamma)$ is
contained in a proper subspace of $H$, contradicting \eqref{de1}). We set:
\[
\cE^n(u,v) :=\int_{H_n} \la\nabla u,\nabla v\ra \, d\gamma_n
\qquad u,\,v\in C^1_b(\Hn ).
\]
We should rather write $\nabla_{\Hn }u$ for $u\in C^1_b(\Hn )$, but
since the scalar product of $\Hn $ is induced by $H$ there is no
ambiguity in writing $\nabla u:\Hn \mapsto \Hn$. We denote by
$(R_\lambda)_{\lambda>0}$, $P_t$ (respectively
$(R_\lambda^n)_{\lambda>0}$, $P^n_t$) the resolvent family and the
semigroup of $\cE_\gamma$ (resp. $\cE^n$). We also know, by the
previous step, that $R_\lambda^n f$ is representable on $C_b(K_n)$
by $\int_0^\infty\exp(-\lambda t)\int_{H_n}f\,d\nu^{n,x}_t\,dt$:
here $K_n$ denotes the support of $\gamma_n$ and $\nu^{n,x}_t$ the
associated Wasserstein semigroup in $\ProbabilitiesTwo{K_n}$. Since
$\gamma_n=(\pi_n)_\#\gamma$, we have $\pi_n(x)\in K_n$ for all $x\in
K$. As a consequence, by Theorem~\ref{stabflows} we obtain that
\[
\lim_{n\to\infty} R_\lambda^n f(\pi_n(x)) =
\lim_{n\to\infty}\int_0^\infty e^{-\lambda t}\int_H f\,
d\nu^{n,\pi_n(x)}_t \, dt= \int_0^\infty e^{-\lambda t} \int_H f \,
d\nu_t^x\, dt \qquad \forall \ x\in K
\]
for all $f\in C_b(H)$. We shall denote, as in Step 1, by $F_\lambda f$ the
right-hand side. Our strategy will be, again, to show that
$R_\lambda f(x)=\int_0^\infty\exp(-\lambda t)F_\lambda f(x)\,dt$.
We assume first that $f=g\circ\pi_k$ is cylindrical function, with
$g\in {\rm Lip}_b(H_k)$; by applying \eqref{felle} to $\gamma_n$ one obtains that
$\lambda [(R_\lambda^n f)\circ\pi_n]_{{\rm Lip}(K)}\leq [f]_{{\rm Lip}(K_n)}
\leq [g]_{\rm Lip(H_k)}$, hence $F_\lambda f\in {\rm Lip}_b(K)$ and
Lemma~\ref{lipb} gives $F_\lambda f\in D(\cE_\gamma)$.

Now, let $v=u\circ\pi_l$ and $u\in C^1_b(H_l)$; for $n\geq\max\{k,l\}$,
taking into account \eqref{iprbis} and the identities $f=f\circ\pi_n$,
$v=v\circ\pi_n$, we have
\begin{eqnarray}\label{lastlast}
\int_H v f \, d\gamma  =
\int_{H_n}  v f \, d\gamma_n &=&
\lambda\int_{H_n} v R_\lambda^n f \,d\gamma_n+\cE^n(v,R_\lambda^n f)\nonumber \\&=&
\lambda\int_H v (R_\lambda^n f)\circ\pi_n\,d\gamma+\cE_\gamma(v,(R_\lambda^n f)\circ\pi_n).
\end{eqnarray}
Now, $(R_\lambda^n f)\circ\pi_n$ converge to $F_\lambda f$ in $L^2(\gamma)$ and
is bounded with respect to the norm $\Vert\cdot\Vert_{\cE_{\gamma,1}}$, by the
uniform Lipschitz bound. Therefore, by the closability of $\cE_\gamma$,
$(R_\lambda^n f)\circ\pi_n\to F_\lambda f$ in the weak topology of $D(\cE)$.
Thus, we can passing to the limit as $n\to\infty$ in \eqref{lastlast} to
obtain
$$
\int_H v f \, d\gamma  =
\lambda\int_H v F_\lambda f\,d\gamma+\cE_\gamma(v,F_\lambda f)
\qquad\forall v=u\circ\pi_l,\,\,u\in C^1_b(H_l).
$$
By the $L^2(\gamma)$ density of $C^1_b$ cylindrical functions $v$,
we obtain $R_\lambda f(x)=F_\lambda f$
for all Lipschitz and bounded cylindrical functions $f$. As a consequence,
\eqref{rdi1} holds for this class of functions. Since both sides in \eqref{rdi1}
are continuous with respect to $\gamma$-almost sure and dominated convergence,
again a density argument shows that the equality \eqref{rdi1}
extends to all $f\in L^\infty(\gamma)$.

By the previous step, we know that \eqref{rdi12} holds for the
finite-dimensional case. By approximation, using the stability
result of Theorem \ref{stabflows}, the contractivity of gradient
flows of Theorem \ref{tgflow} and Lemma \ref{easy} below, we obtain
the general case.
\end{proof}

In the proof of Theorem~\ref{nonsmooth} we also used the following result.
\begin{lemma}\label{easy}
Let $\sigma_n,\,\sigma_\infty$ be finite signed measures on $H$ and
let $\varphi_n,\,\varphi_\infty:H\to\R$ satisfy:
\begin{itemize}
\item[(i)] $\sup_n|\sigma_n|(H)<+\infty$ and
\[
\lim_{n\to\infty}\int\varphi\, d\sigma_n=
\int\varphi\, d\sigma_\infty\qquad \forall\varphi\in C_b(H);
\]
\item[(ii)] there exist compacts sets $J_m\subset H$ such that
$\sup_n |\sigma_n|\left(H\setminus J_m\right)\to 0$ as $m\to\infty$;
\item[(iii)] $\{\varphi_n\}_{n\in\N\cup\{\infty\}}$ is
equi-bounded and equi-continuous;
\item[(iv)] $\varphi_n$ converge pointwise to $\varphi_\infty$ on
${\rm supp\,}\sigma_\infty$.
\end{itemize}
Then $\lim_n\int_H\varphi_n\,d\sigma_n=\int_H\varphi_\infty\,d\sigma_\infty$.
\end{lemma}
\noindent
\begin{proof} Without loss of generality we can assume that $\lim_n\int_H\varphi_n\,d\sigma_n$
exists (so that we can freely extract subsequences) and $|\varphi_n|\leq 1$, $|\varphi_\infty|\leq 1$.
Let us fix $m$ and assume, possibly extracting a subsequence, that $\varphi_n\to\psi_m$ uniformly
on $J_m$ as $n\to\infty$; obviously $\varphi_\infty=\psi_m$ on $J_m\cap{\rm supp\,}\sigma_\infty$.
We extend $\psi_m$ continuously to the whole of $H$ with $|\psi_m|\leq 1$. Then:
\begin{eqnarray*}
\left|\int_H \varphi_n \, d\sigma_n - \int_H \varphi_\infty \, d\sigma_\infty\right|
&\leq&\left|\int_H\varphi_n \,d\sigma_n- \int_H\psi_m\, d\sigma_n\right|+
\left|\int_H \psi_m\,d\sigma_n-\int_H \psi_m\,d\sigma_\infty\right|\\&+&
\left|\int_H \psi_m\,d\sigma_\infty-\int_H \varphi_\infty\,d\sigma_\infty\right|.
\end{eqnarray*}
The first term in the right hand side can be estimated, splitting the integration on $J_m$
and on $H\setminus J_m$, with
$\sup_{J_m}|\varphi_n-\psi_m||\sigma_n|(H)+2|\sigma_n|(H\setminus J_m)$.
The second term tends to 0 as $n\to\infty$ by
our first assumption, while the third one can be estimated with $2|\sigma_\infty|(H\setminus J_m)$.
Therefore, taking first the limsup as $n\to\infty$ and then letting $m\to\infty$ we
have the thesis.
\end{proof}

\begin{remark}[Continuity of $P_t$]
{\rm The $\ProbabilitiesTwo{H}$-continuity of
$t\mapsto\Semi{\delta_x}{t}$ in $[0,+\infty[$ shows that $P_tf\to f$
pointwise in $K$ as $t\downarrow 0$ for all functions $f\in C(K)$
with at most quadratic growth at infinity, and in particular for
$f\in C_b(K)$. Taking \eqref{felle} into account, the convergence is
uniform on compact subsets of $H$ if $f\in {\rm Lip}_b(K)$; by
density, $P_tf\to f$ uniformly on compacts sets as $t\downarrow 0$
for all $f\in UC_b(K)$, the space of bounded \emph{uniformly}
continuous functions on $H$. It is also possible to show the
regularizing effect $P_t(C_b(K))\subset UC_b(K)$ for $t>0$: indeed,
the finite-dimensional smooth systems are easily seen to be Strong
Feller (see section 7.1 of \cite{dpz2}), and this property extends
to the general case by approximation. }\end{remark}

\section{The Markov process}\label{process}

In this section we complete the proofs of Theorem~\ref{main1},
Theorem~\ref{main2}, Theorem~\ref{main3} and Theorem~\ref{main4},
proving the existence of a unique Markov family of probability
measures $\bbP_x$ on $K^{[0,+\infty[}$ satisfying
\begin{equation}
P_tf(x)=\E_x[f(X_t)]\qquad\forall x\in K
\end{equation}
for all bounded Borel functions $f$. The continuity of $\bbP_x$ will
be a consequence of the regularizing properties of the Wasserstein
semigroup $\Semi{\mu}{t}$ (in particular the continuity of
$x\mapsto\Semi{\delta_x}{t}$ will play an important role). The
regularity property \eqref{outer}, instead, is based on general
results from \cite{maro}, that provide a Markov family satisfying a
weaker property, and on the continuity of
$t\mapsto\Semi{\delta_x}{t}$. As in the previous section we will use
the notation $\nu_t^x$ for $\Semi{\delta_x}{t}$.

Recall also that the regularizing estimate (iii) in
Theorem~\ref{tgflow} give $\nu_t^x\ll\gamma$ for all $x\in K$; by
the uniform bound on the relative Entropy (which yields
equi-integrability of the densities), Dunford-Pettis theorem
provides the continuity property
\begin{equation}\label{privico}
x_n\in K, \quad \|x_n-x\|\to 0, \quad \rho_n\gamma=\nu_t^{x_n},
\quad \rho\gamma=\nu_t^x \quad\Longrightarrow\quad
\text{$\rho_n\rightharpoonup\rho$ weakly in $L^1(\gamma)$}
\end{equation}
for all $t>0$.

\smallskip
\noindent {\bf Proof of Theorem~\ref{main1}.} We already proved
statement (a) in Lemma~\ref{lipb}. Let us consider the semigroup
$P_t$ induced by $\cE_\gamma$, linked to $\nu_t^x$ by \eqref{rdi1};
the semigroup property of $P_t$ can be read at the level of
$\nu_t^x$, and gives the Chapman-Kolmogorov equations. Therefore
these measures are the transition probabilities of a
time-homogeneous Markov process $\bbP_x$ in $K^{[0,+\infty[}$. In
particular, the Markov property gives the explicit formula
\begin{eqnarray}\label{fdd}
&&\bbP_x(\{X_{t_1}\in A_1\}\cap\cdots\cap\{X_{t_{n-1}}\in A_{n-1}\}\cap\{X_{t_n}\in A_n\})\\&=&
\int_{A_1}\cdots\int_{A_{n-1}}\int_{A_n}
1\,d\nu_{t_n-t_{n-1}}^{y_{n-1}}(y_n)\,d\nu^{y_{n-2}}_{t_{n-1}-t_{n-2}}(y_{n-1})
\cdots d\nu^{y_0}_{t_1-t_0}(y_1)\nonumber
\end{eqnarray}
(with $y_0=x$, $0=t_0\leq t_1<\cdots<t_n<+\infty$, $A_1,\ldots,A_n\in\BorelSets{K}$)
for these finite-dimensional distributions. The continuity of $x\mapsto\bbP_x$, namely
the continuity of all finite-dimensional distributions, is a direct consequence of
\eqref{privico} and \eqref{fdd}.

Now, let us prove \eqref{outer}. In order to apply the general results of \cite{maro}, we need
to emphasize two more properties of $\cE_\gamma$. First, $\cE_\gamma$ is \emph{tight}:
this means that there exists a nondecreasing sequence of compact sets $F_m\subset H$ such that
${\rm cap}_\gamma(H\setminus F_m)\to 0$ ( ${\rm cap}_\gamma$ being the capacity induced by $\cE_\gamma$, see
\cite{maro}). This can be proved using \eqref{Markov} and the
argument in \cite[Proposition~IV.4.2]{maro}: let $(x_n)\subset H$ be a dense sequence and define
$$
w_n(x):=\min\left\{1,\min_{0\leq i\leq n}\|x-x_i\|\right\}.
$$
It is immediate to check that $0\leq w_n\leq 1$, $w_n\downarrow 0$ in $H$ and
$[w_n]_{\rm Lip(H)}\leq 1$. Therefore $(w_n)$ is bounded in the weak topology of
$D(\cE_\gamma)$ and converges to $0$ in the weak topology of $\cE_\gamma$.
The Banach-Saks theorem ensures the existence of a subsequence
$(n_k)$ such that the Cesaro means
$$
v_k:=\frac{w_{n_1}+\cdots+w_{n_k}}{k}
$$
converge to $0$ strongly in $D(\cE_\gamma)$. This implies \cite[Proposition~III.3.5]{maro}
that a subsequence $(v_{k(l)})$ of $(v_k)$ converges to $0$ quasi-uniformly, i.e. for all integers $m\geq 1$
there exists a closed set $G_m\subset H$ such that ${\rm cap}_\gamma(H\setminus G_m)<1/m$
and $v_{k(l)}\to 0$ uniformly on $G_m$. As $w_{n(k(l)))}\leq v_{k(l)}$, if we set $F_m=\cup_{i\leq m}G_i$, we
have that $w_{n(k(l))}\to 0$ uniformly on $F_m$ for all $m$ and
${\rm cap}_\gamma(H\setminus F_m)\leq 1/m$. If $\epsilon>0$ and $n$ is an integer such that
$w_n<\epsilon$ on $F_m$, the definition of $w_n$ implies
$$
F_m\subset\bigcup_{i=1}^n B(x_i,\epsilon).
$$
Since $\epsilon$ is arbitrary this proves that $F_m$ is totally bounded, hence compact.
This completes the proof of the tightness of $\cE_\gamma$.

Second, $\cE_\gamma$ is \emph{local},
i.e. $\cE_\gamma(u,v)=0$ whenever $u,\,v\in D(\cE_\gamma)$ have compact and disjoint
support. This can be easily achieved (see also \cite[Lemma~V.1.3]{maro}) taking sequences
$(u_n),\, (v_n)\subset C^1_b(A(\gamma))$ converging to
$u$ and $v$ respectively in the norm $\|\cdot\|_{\gamma,1}$,
and modifying them, without affecting the convergence, so that
$u_n$ and $v_n$ have disjoint supports. One concludes noticing
that $\cE_\gamma(f,g)=0$ whenever $f,\,g\in C^1_b(A(\gamma))$ have
disjoint supports.

These properties imply, according to \cite[Theorem~IV.3.5,
Theorem~V.1.5]{maro} the existence of a Markov family of probability
measures $\{\tilde{\bbP}_x\}_{x\in K}$ on $C([0,+\infty[;K)$
(uniquely determined up to $\gamma$-negligible sets), satisfying
\begin{equation}\label{maro1}
P_tf(x)=\tilde{\E}_x[f(X_t)]\qquad\text{for $\gamma$-a.e. $x\in K$}
\end{equation}
for all bounded Borel functions $f$ on $K$ (here $\tilde{\E}_x$ is
the expectation with respect to $\tilde{\bbP}_x$). Now, since $H$ is
separable we can find a countable family $\mathcal A$ of open sets
stable under finite intersections which generates $\BorelSets{H}$;
choosing $f=1_A$ in \eqref{maro1} and \eqref{rdi1}, and taking into
account that $\mathcal A$ is countable, we can find a
$\gamma$-negligible set $N\subset K$ such that
$\nu_t^x(A)=\E_x[1_A(X_t)]$ for all $A\in\mathcal A$, $t\in\Q$ and
all $x\in K\setminus N$. As a consequence, $\nu^t_x$ is the law of
$X_t$ under $\tilde{\bbP}_x$ for all $x\in K\setminus N$ and all
$t\in\Q$. We can now use the continuity of the process $X_t$ and of
$t\mapsto\nu_t^x$ to obtain that $\nu_t^x$ are the one-time
marginals of $\tilde{\bbP}_x$ for all $x\in K\setminus N$.

We prove now path continuity under  $\bbP_x$ for $x\in K\setminus
N$, using the property $\nu^x_t\ll\gamma$: we adapt the approach of
\cite{do} to our setting. By the Markov property we obtain that all
finite-dimensional distributions of $\tilde{\bbP}_x$ and $\bbP_x$
coincide; as a consequence, if we denote by $i:\Omega\to
K^{[0,+\infty[}$ the (obviously measurable) injection map,
$i_\#\tilde{\bbP}_x=\bbP_x$. By the Ulam lemma, we can find compacts
sets $K_n\subset\Omega$ with $\tilde{\bbP}_x(K_n)\uparrow 1$; now,
if $S\subset [0,+\infty)$ is bounded, countable and $B_S\subset
K^{[0,+\infty[}$ is the measurable set defined by
$$
B_S:=\left\{\omega\in\Omega:\ \text{the restriction of $\omega$ to $S$ is
uniformly continuous}\right\},
$$
from the inclusion $B_S\cap\Omega\supset K_n$ we obtain $\bbP_x(B_S)\geq\tilde{\bbP}_x(K_n)$,
hence $\bbP_x(B_S)=1$. A well known criterion \cite[Lemma~2.1.2]{strook} then gives that
$\bbP_x^*(\Omega)=1$. This proves that
\begin{equation}\label{outer}
\bbP_x^*(\Omega)=1\qquad\text{for $\gamma$-a.e. $x\in K$.}
\end{equation}

In order to show the first part of statement (c), fix $x\in K$,
and define $B_S$ as above, with $S\subset ]0,+\infty[$ satisfying $\varepsilon:=\inf S>0$ and
$\sup S<\infty$. Since the law of $X_\varepsilon(x)$ is absolutely continuous with respect to
$\gamma$, we know from \eqref{outer} that $\bbP_{X_\varepsilon}(B_{S-\varepsilon})=1$ $\bbP_x$-almost surely.
Taking expectations, and using the Markov property, we get $\bbP_x(B_S)=1$. Again the same argument
in \cite[Lemma~2.1.2]{strook} shows that $\bbP_x^*\left(C(]0,+\infty[;H)\right)=1$.

Finally, we use the representation \eqref{rdi1} and the fact that $P_t$ is selfadjoint
(due to the fact that $\cE_\gamma$ is symmetric) to obtain
\begin{equation}\label{caserta}
\int_K\int_K \varphi(x)\psi(y)\,d\nu^y_t(x)\,d\gamma(y)=
\int_K\int_K \varphi(x)\psi(y)\,d\nu^x_t(y)\,d\gamma(x)
\qquad\forall \varphi,\,\psi\in L^\infty(\gamma).
\end{equation}
This means that the process $\bbP_x$ is reversible.

\smallskip
\noindent {\bf Proof of Theorem~\ref{main2}.} It is a direct
consequence of the estimates in Theorem~\ref{tgflow} and of the
coincidence, proved above, of the law of $X_t$ under $\bbP_x$ with
$\nu_t^x$. {\penalty-20\null\hfill$\square$\par\medbreak}

\smallskip
\noindent {\bf Proof of Theorem~\ref{main3}.} We shall denote by
$\nu^{n,x}_t$, $\E_x^n$ (resp. $\nu^x_t$, $\E_x$) the transition
probabilities and the expectations relative to $\bbP_x^n$ (resp.
$\bbP_x$). From Theorem~\ref{stabflows} we obtain:
\begin{equation}\label{privico1}
x_n\in K_n,\,\,\Vert x_n-x\Vert\to 0,\,\,x\in
K\quad\Longrightarrow\quad \text{$\nu^{n,x_n}_t\to\nu^x_t$ in
$\ProbabilitiesTwo{H}$.}
\end{equation}
We shall prove by induction on $m$ that $\E_{x_n}^n[f(X_{t_1},\ldots,X_{t_m})]\to
\E_x[f(X_{t_1},\ldots,X_{t_m})]$ for all $f\in C_b(H^m)$. Obviously we can
restrict ourselves to $f\in {\rm Lip}_b(H^m)$ and the case $m=1$ corresponds
to \eqref{privico1}. So, let us assume the statement valid for $m\geq 1$ and let us
prove it for $m+1$. Let $f\in {\rm Lip}_b(H^{m+1})$, and let $\Pi_n:H\to K_n$
be the canonical projection. By the weak convergence of $\gamma_n$ to $\gamma$,
we have $\|\Pi_n(y)-y\|\to 0$ for all $y\in K$. As a consequence, the induction
assumption gives $\E^n_{\Pi_n(y)}[f(y,X_{t_2},\ldots,X_{t_{m+1}})]\to
\E_y[f(y,X_{t_2},\ldots,X_{t_{m+1}})]$. Since $f$ is Lipschitz
we have also
\begin{equation}\label{triu}
\lim_{n\to\infty}
\E^n_{\Pi_n(y)}[f(\Pi_n(y),X_{t_2},\ldots,X_{t_{m+1}})]=
\E_y[f(y,X_{t_2},\ldots,X_{t_{m+1}})]\qquad\forall y\in K.
\end{equation}
Thanks to \eqref{triu} and Lemma~\ref{easy}, we can pass to the limit as $n\to\infty$
in the identity
\begin{eqnarray*}
\E^n_x[f(X_{t_1},\ldots,X_{t_{m+1}})]&=&\int_K\E^n_y[f(y,X_{t_2},\ldots,X_{t_{m+1}})]\,d\nu^{n,x_n}_{t_1}(y)
\\&=&\int_H\E^n_{\Pi_n(y)}[f(\Pi_n(y),X_{t_2},\ldots,X_{t_{m+1}})]\,d\nu^{n,x_n}_{t_1}(y)
\end{eqnarray*}
to obtain $\E^n_x[f(X_{t_1},\ldots,X_{t_{m+1}})]\to
\int_K\E_y[f(y,X_{t_2},\ldots,X_{t_{m+1}})]\,
d\nu^x_{t_1}(y)=\E_x[f(X_{t_1},\ldots,X_{t_{m+1}})]$. This proves
statement (a). Statements (b) and (c) follow at once by the
tightness Lemma~\ref{cla3} below.
{\penalty-20\null\hfill$\square$\par\medbreak}

\begin{lemma}[Tightness]\label{cla3}
Let $\gamma_n$ and $\gamma$ as in Theorem~\ref{main3}, let $x\in
K(\gamma)$ and let $x_n\in K(\gamma_n)$ be such that $x_n\to x$. For
all $0<\varepsilon\leq T<+\infty$, $h\in H$, 
the laws of $(\langle X_t,h\rangle_H, t\in[\varepsilon,T])$ under
$\bbP^n_{x_n}$, $n\in\N$, form a
tight sequence in $C([\varepsilon,T])$. Moreover the laws of
$(\langle X_t,h\rangle_H, t\in[0,T])$ under $\bbP^n_{\gamma_n}$,
$n\in\N$, form a tight sequence in $C([0,T])$.
\end{lemma}
\begin{proof} Let $H_n=H^0(\gamma_n)$, $K_n=K(\gamma_n)$ and
$\bbP^n_{\gamma_n}=\int_{H_n}\bbP_x^n\,d\gamma_n(x)$.
For any $h\in H_n$ we have by the Lyons-Zheng
decomposition, see e.g. \cite[Th. 5.7.1]{fot} that, under $\bbP^n_{\gamma_n}$,
\[
\la h,X_t-X_0\ra_{H_n}
\, = \, \frac 12 \,
M_t \, - \, \frac 12 \, (N_T  - N_{T-t}),\qquad\forall t\in [0,T],
\]
where $M$, respectively $N$, is a $\bbP^n_{\gamma_n}$-martingale with respect to
the natural filtration of $(X_t,\ t\in [0,T])$, respectively of $(X_{T-t}, \ t\in[0,T])$.
Moreover, the quadratic variations $\langle M\rangle_t$, $\langle N\rangle_t$ are both
equal to $t \cdot \|h\|^2_{H_n}$. By the Burkholder-Davis-Gundy inequality we can
find, for all $p>1$, a constant $c_p\in (0,+\infty)$ such that
\begin{equation}\label{cla2}
{\mathbb E}^n_{\gamma_n}\left[\left|\la h,X_t-X_s\ra_{H_n}\right|^p\right]
\leq c_p \|h\|^p_{H_n} |t-s|^{p/2}, \qquad t,\,s\in[0,T].
\end{equation}
Let us denote by $\Pi_n:H\to H_n$ the duality map satisfying
$\langle\Pi_n(h),v\rangle_{H_n}=\langle h,v\rangle_H$ for all
$h\in H_n$. Then, choosing $v=\Pi_n(h)$, from \eqref{bastakappa}
we get $\Vert\Pi_n(h)\Vert_{H_n}\leq\kappa\Vert h\Vert_H$,
so that \eqref{cla2} gives
\[
{\mathbb E}^n_{\gamma_n}\left[\left|\la h,X_t-X_s\ra_{H}\right|^p \right]
\leq\kappa^p  c_p \|h\|_H^p |t-s|^{p/2}, \qquad t,\,s\in[0,T].
\]
Then tightness of the laws of $\langle X_t,h\rangle_H$ under
$\bbP^n_{\gamma_n}$ in $C([0,T])$ follows e.g. by
\cite[Exercise~2.4.2]{strook}.

\smallskip
Let $\varepsilon>0$ and let us prove that
$\sup_n\RelativeEntropy{\nu_\varepsilon^{n,x_n}}{\gamma_n}<+\infty$.
Since $x\in K$ there exist $R_\varepsilon>0$ such that
$\gamma(B_{R_\varepsilon}(x))>1/2$, so that there exists
$n_\varepsilon$ such that $\gamma_n(B_{R_\varepsilon}(x_n))\geq 1/2$
for all $n\geq n_\varepsilon$. Let $\nu_n:=\gamma_n(\cdot|
B_{R_\varepsilon}(x_n))$; then for $n\geq n_\varepsilon$ we have
$W^2_{2,H_n}(\delta_{x_n},\nu_n) \leq \kappa^2
W_2^2(\delta_{x_n},\nu_n)\leq 2\kappa^2R^2$, and since
$\RelativeEntropy{\nu_n}{\gamma_n}=-\ln\gamma_n(B_{R_\varepsilon}(x_n))$,
by (iii) of Theorem \ref{tgflow} we get
\[
\RelativeEntropy{\nu_\varepsilon^{x_n}}{\gamma_n}\leq
\frac{2\kappa^2R^2}\varepsilon  + \ln 2\qquad\forall n\geq
n_\varepsilon.
\]
Let ${\bf P}_n$ (resp. $\overline {\bf P}_n$) be the law of $(X_t,t\in[\varepsilon,T])$ under
$\bbP_{x_n}^n$ (resp. $\bbP^n_{\gamma_n}$).
Let us prove that ${\bf P}_n\ll\overline {\bf P}_n$ and
\begin{equation}\label{srotondo}
\frac{d{\bf P}_n}{d\overline {\bf P}_n}  = \rho^n_\varepsilon(X_\varepsilon)
\qquad\text{$\overline {\bf P}_n$-almost surely,}
\end{equation}
where $\rho^n_\varepsilon$ is the density of
$\nu_\varepsilon^{n,x_n}$ with respect to $\gamma_n$. For any
bounded and Borel functional $\Phi:C([0,T-\varepsilon];H)\mapsto\R$,
we have by the Markov property:
\[
\begin{split}
\E_{x_n}^n(\Phi(X_{\varepsilon+\cdot})) & = \E^n_{x_n}(
\E_{X_\varepsilon}^n(\Phi)) = \int d\nu_\varepsilon^{n,x_n}(y) \,
\E_{y}^n(\Phi) = \int d\gamma_n(y) \, \rho^n_\varepsilon(y) \,
\E_{y}^n(\Phi) \\ & = \E_{\gamma_n}^n(\rho^n_\varepsilon(X_0) \,
\Phi(X_\cdot)) = \E_{\gamma_n}^n(\rho^n_\varepsilon(X_\varepsilon)
\, \Phi(X_{\varepsilon+\,\cdot})),
\end{split}
\]
where in the last equality we use stationarity, and \eqref{srotondo}
is proven.

Let now $h\in H$ and let ${\bf P}_n^h$ (resp. $\overline {\bf
P}_n^h$) be the law of $(\langle
X_t,h\rangle_H,t\in[\varepsilon,T])$ under $\bbP_{x_n}^n$ (resp.
$\bbP^n_{\gamma_n}$); since the relative Entropy does not increase
under marginals \cite[9.4.5]{ags}, from \eqref{srotondo} we get
$$
\RelativeEntropy{{\bf P}_n^h}{\overline {\bf P}^h_n}\leq
\RelativeEntropy{{\bf P}_n}{\overline {\bf P}_n}=
\RelativeEntropy{\nu_\varepsilon^{x_n}}{\gamma_n}.
$$
It follows that $\sup_n\RelativeEntropy{{\bf P}_n^h}{\overline {\bf
P}^h_n}$ is finite. By applying the entropy inequality
\eqref{entrine}, tightness of $(\overline {\bf P}^h_n)$ implies
tightness of $({\bf P}^h_n)$.
\end{proof}

\smallskip
\noindent {\bf Proof of Theorem~\ref{main4}.} Let
$\mu\in\ProbabilitiesTwo{H}$ be an invariant measure for
$(P_t)_{t\geq 0}$. Then, by \eqref{rdi12}, $\Semi{\mu}{t}\equiv\mu$
is a constant gradient flow of $\RelativeEntropy{\cdot}{\gamma}$ and
therefore, by \eqref{EVI},
$\RelativeEntropy{\mu}{\gamma}\leq\RelativeEntropy{\nu}{\gamma}$ for
all $\nu\in\ProbabilitiesTwo{H}$. Since $t\mapsto t\ln t$ is
strictly convex, the unique minimizer of
$\RelativeEntropy{\cdot}{\gamma}$ in $\ProbabilitiesTwo{H}$ is
$\gamma$, and therefore $\mu=\gamma$.
{\penalty-20\null\hfill$\square$\par\medbreak}

\appendix
\section{Some properties of log-concave measures}
\label{properties}

In this appendix we state and prove some useful properties of log-concave
measures and of convex functions used throughout the paper.

First of all, for lower semicontinuous convex functions
$V:\R^k\to\R\cup\{+\infty\}$ (i.e. the typical densities of
log-concave measures),  we recall that the properties
$\int\exp(-V)\,dx<+\infty$, $V(x)\to +\infty$ as $\|x\|\to +\infty$
and $V(x)\to +\infty$ at least linearly as $\|x\|\to +\infty$ are
all equivalent: indeed, the equivalence between the second and the
third one simply follows by the monotonicity of difference quotients
along radial directions, and clearly a linear growth at infinity
implies finiteness of the integral. On the other hand, if the
integral is finite, a crude growth estimate on $V$ can be obtained
as follows: assuming with no loss of generality that $\{V<+\infty\}$
has nonempty interior, we can find a ball $B$ and $M<+\infty$ such
that $V\leq M$ on $B$; then, on the convex cone $C_x$ generated by
$x$ and $B$, we have the inequality $V\leq M+V^+(x)$. Changing signs
and taking exponentials we can integrate on $C_x$ to obtain
$$
e^{V^+(x)+M}\geq\left(\int_{\R^k}e^{-V(y)}\,dy\right)^{-1}\Leb{k}(C_x)\to +\infty
\quad\text{as $|x|\to +\infty$.}
$$

\begin{lemma}[Variational characterization of $\cE_\gamma(u,u)$]\label{slopeentro}
Let $\gamma\in\ProbabilitiesTwo{H}$ be a non-degenerate log-concave
measure, and let $u$ be a bounded $C^1$ cylindrical function, with
$\inf u>0$ and $\int u^2\,d\gamma=1$. Then $\sqrt{\cE_\gamma(u,u)}$
is the smallest constant $S$ satisfying
\begin{equation}\label{beppe1}
\RelativeEntropy{\mu}{\gamma}\geq\RelativeEntropy{u^2\gamma}{\gamma}-
2 \, S \, W_2(\mu,u^2\gamma),
\qquad\forall\mu\in\ProbabilitiesTwo{H}.
\end{equation}
\end{lemma}
\begin{proof} First, we realize that this is essentially a finite-dimensional statement.
Indeed, if $\pi:H\to L$ is a finite-dimensional orthogonal projection such that
$u=u\circ\pi$, a simple application of Jensen's inequality gives \cite[Lemma~9.4.5]{ags}
$\RelativeEntropy{\mu}{\gamma}\geq\RelativeEntropy{\pi_\#\mu}{\pi_\#\gamma}$.
Since $W_2(\pi_\#\mu,u^2\pi_\#\gamma)=W_2(\pi_\#\mu,\pi_\#(u^2\gamma))\leq
W_2(\mu,u^2\gamma)$ and $\cE_\gamma(u,u)=\cE_{\pi_\#\gamma}(u,u)$, we need
only to check the analog of \eqref{beppe1} with
$\gamma$ replaced by $\pi_\#\gamma$ and $H$ replaced by $L$. So, from now
on we shall assume that $H=\R^k$ for some integer $k$.

In Lemma~\ref{lenergyi} we proved that
$$
\RelativeEntropy{\mu}{\gamma}\geq\RelativeEntropy{u^2\gamma}{\gamma}-
2\sqrt{\cE_\gamma(u,u)}\, W_2(\mu,u^2\gamma),
\qquad\forall\mu\in\ProbabilitiesTwo{\R^k}
$$
when $\gamma=\exp(-V)\Leb{k}$ with $V$ smooth, convex, and $\nabla
V$ and all its derivatives are bounded (it suffices to use the
Schwartz inequality to estimate from below the scalar product in
\eqref{energy}). By monotone approximation (see Step~2 in the proof
of Theorem~\ref{nonsmooth}) the same inequality holds for all
log-concave $\gamma$ in $\R^k$.

It remains to show that $\sqrt{\cE_\gamma(u,u)}$ is the smallest constant with this
property. In order to prove this fact, we fix $\ss\in C^\infty_c(\R^k;\R^k)$ with support
contained in the interior of $\{V<+\infty\}$, and consider the maps
$\tt_\varepsilon:=\ii+\varepsilon\ss$ and the measures
$\mu_\varepsilon=(\tt_\varepsilon)_\#(u^2\gamma)$, so that
$W_2^2(\mu_\varepsilon,u^2\gamma)\leq\varepsilon^2\int \|\ss\|^2 u^2\,d\gamma$.
On the other hand, the area formula gives that the density of $\mu_\varepsilon$ with
respect to $\Leb{k}$ is given by $f_\varepsilon$, where
$$
f_\varepsilon:=\frac{u^2\exp(-V)}{|{\rm \det}\nabla\tt_\varepsilon|}\circ \tt_\varepsilon^{-1}
$$
(notice that for $\varepsilon$ small enough $\tt_\eps$ is a diffeomorphism which leaves
$\{V<+\infty\}$ invariant). Since
\begin{eqnarray*}
\RelativeEntropy{\mu_\varepsilon}{\gamma}&=&
\int_{\R^k} f_\varepsilon\ln f_\varepsilon\,dx+\int_{\R^k} f_\varepsilon V\,dx=
\int_{\R^k} \ln (f_\varepsilon\circ\tt_\varepsilon) u^2\,d\gamma
+\int_{\R^k} (V\circ\tt_\varepsilon) u^2\,d\gamma\\&=&
\RelativeEntropy{u^2\gamma}{\gamma}-\varepsilon\int_{\R^k}(\nabla\cdot\ss)u^2\exp(-V)-
\langle\nabla V,\ss\rangle u^2\exp(-V)\,dx+o(\varepsilon)\\
&=&
\RelativeEntropy{u^2\gamma}{\gamma}+2\varepsilon\int_{\R^k} u
\langle\nabla u,\ss\rangle\,d\gamma+o(\varepsilon)
\end{eqnarray*}
from \eqref{beppe1} we get
$$
\int_{\R^k}u\langle\nabla u,\ss\rangle\,d\gamma\geq - S\sqrt{\int_{\R^k} \|\ss\|^2 u^2\,d\gamma}.
$$
Since $u^2\gamma$ is concentrated in the interior of $\{V<+\infty\}$, we can approximate
in $L^2(u^2\gamma;\R^k)$ the function $-\nabla\ln u$ with $\ss$ it follows that
$S\geq\sqrt{\cE_\gamma(u,u)}$.
\end{proof}

In the next two propositions, borrowed essentially from \cite{za}, we show that,
for convex functions $U:\R^k\to\R$, the growth at infinity of $\nabla U$ is always balanced
by the factor $e^{-U}$; this leads to uniform bounds and tightness estimates
for the measures $|\nabla U|e^{-U}\Leb{k}$, under uniform lower bounds on $U$.

\begin{proposition}\label{pzamb1}
Let $U:\R^k\to\R\cup\{+\infty\}$ be convex and lower semicontinuous,
with $U(x)\to+\infty$ as $\|x\|\to +\infty$, $\{U<+\infty\}$ having
a nonempty interior, and set $\gamma=\exp(-U)\Leb{k}$. Then, for all
unit vectors $h\in\R^k$ there exists a unique finite signed measure
$\Sigma_h^U$ in $\R^k$ supported on $\overline{\{U<+\infty\}}$ such
that
\begin{equation}\label{defsigmav}
\int_{\R^k}\frac{\partial u}{\partial h}\,d\gamma=\int_{\R^k}u\,d\Sigma_h
\qquad\forall u\in C^1_b(\R^k).
\end{equation}
Moreover, we have
$|\Sigma_h^U|(\R^k)=2\int_{h^\perp}\exp(-\min\limits_{t\in\R}U(y+th))\,dy$.
\end{proposition}
\begin{proof}
Assume first $k=1$; the function $t\mapsto\exp(-U(t))$ is infinitesimal at
infinity, non-decreasing on a half-line $(-\infty,t_0)$ and non-increasing on $(t_0,+\infty)$,
where $t_0$ is any point in the interior of $\{U<+\infty\}$
where $U$ attains its minimum value.
Then $\exp(-U)$ has bounded variation on $\R$ and the total
variation of its distributional derivative $\frac{d}{dt}\exp(-U)$
is representable by:
\[
\left|\frac{d}{dt} \, e^{-U}\right| =
1_{(t<t_0)} \frac{d}{dt} e^{-U} - 1_{(t>t_0)} \frac{d}{dt} e^{-U}.
\]
It follows that $|\frac{d}{dt}\exp(-U)|(\R)=2\exp(-U(t_0))=2\exp(-\min\limits_{\R}U)$;
by definition of distributional derivative, $\Sigma^U=-\frac{d}{dt}\exp(-U)$ fulfils
\eqref{defsigmav} when the function $u$ is compactly supported, and a simple density
argument gives the general case.

In the case $k>1$ we denote $U_y(t):=U(y+th)$; since $U$ has at least linear
growth at infinity, it is easy to check that $\exp(-\min\limits_\R U_y)$ is
integrable on $h^\perp$. Now, notice that
Fubini's theorem implies the existence of
$\Sigma^U_h$ and its coincidence with the measure
$\int_{h^\perp}\Sigma^{U_y}\,dy$, i.e.
$$
\int_{\R^k}\frac{\partial u}{\partial h}e^{-U}\,dy \, dt=
\int_{h^\perp}\left(\int_{\R}\frac{d}{dt}u(y+th)e^{-U_y(t)}\,dt\right)\,dy,
\qquad u\in C^1_b(\R^k).
$$
On the other hand, if we denote by $A$ the projection on $h^\perp$ of the interior of
the convex set $\{U<+\infty\}$, and by $C$ the projection of $\{U<+\infty\}$, we have that
$\{U_y<+\infty\}$ has nonempty interior for all $y\in A$, while $U_y$ is identically equal
to $+\infty$ for all $y\in h^\perp\setminus C$; points $y$ in $C\setminus A$ correspond
to projections of boundary points of $\{U<+\infty\}$ where $h$ is tangential to the boundary,
and the co-area formula gives that this set of points is $\Leb{k-1}$-negligible in $h^\perp$.
As a consequence, $|\Sigma^{U_y}|(\R)=2\exp(-\min U_y)$ for $\Leb{k-1}$-a.e. $y\in h^\perp$.
A general result \cite[Corollary~2.29]{afp}
allows to commute total variation and integral, so that
$$
|\int_{h^\perp}\Sigma^{U_y}\,dy|(\R^k)=
\int_{h^\perp}|\Sigma^{U_y}|(\R)\,dy=
2\int_{h^\perp}\exp(-\min\limits_{t\in\R}U(y+th))\,dy.
$$
\end{proof}

\begin{proposition}[Continuity and tightness]\label{pzamb2}
Let $V_n:\R^k\to\R\cup\{+\infty\}$ be convex and lower
semicontinuous function, with $V_n\uparrow V$ and
$\int\exp(-V_1)<+\infty$. Then for all unit vectors $h\in\R^k$ there
exist compact sets $J_m\subset\R^k$ such that:
\begin{equation}\label{sig2}
|\Sigma^{V_n}_h|(\R^k\setminus J_m)\leq\frac 1 m\qquad \forall \ n,\,m\geq 1.
\end{equation}
Furthermore, $\Sigma_h^{V_n}\to\Sigma^V_h$ in the duality with $C_b(\R^k)$.
\end{proposition}
\begin{proof} Let $A\in\R$, $B>0$ be such that $V_1(x)\geq A+B\|x\|$ for
all $x\in\R^k$. We set $\Sigma^n:=\Sigma^{V_n}_h$, $\Sigma:=\Sigma^V_h$.  We
first notice that $V_n(y+th)\geq V_1(y+th)\geq A+B\|y\|$
for all $y\in h^\perp$. Therefore, taking into account the representation of $|\Sigma^U_h|(\R^k)$
given by the previous proposition, we obtain that $|\Sigma^n|(\R^k)$ is uniformly
bounded. On the other hand, since $\exp(-V_n)\Leb{k}$ weakly converge to $\exp(-V)\Leb{k}$
(by the dominated convergence theorem) from \eqref{defsigmav} we infer that
$\Sigma^n\to\Sigma$ weakly in the duality with $C_c^1(\R^k)$, and then in the
duality with $C_c(\R^k)$.

We will prove that
\begin{equation}\label{cl}
\lim_{n\to\infty} |\Sigma^n|(\R^k)  = |\Sigma|(\R^k).
\end{equation}
Before proving (\ref{cl}), we show that it implies (\ref{sig2}):
consider a dense sequence $(x_j)$ in $\R^k$ and set, for $p,\,l\geq
1$ integers, $A_l^p:=\cup_{j=1}^l B(x_j,1/p)$. It is enough to prove
that for all $p$ there exists $l=l(p)$ such
$|\Sigma^n|(\R^k\setminus A_l^p)\leq 2^{-p}/m$ for all $n$: indeed,
in this case $J_m:=\cap_p \overline{A}_{l(p)}^p$ is a compact set
such that $|\Sigma^n|(\R^k\setminus J_m)\leq 1/m$ for all $n\geq 1$.
If, for some $p$, we can not find such $l$, then for all $l$ there
exists $n(l)$ such that $|\Sigma^{n(l)}|(A_l^p)\leq
|\Sigma_h^{n(l)}|(\R^k)-2^{-p}/m$. Since $n(l)$ must tend to
$+\infty$ as $l\to\to\infty$, and any open ball $B_r(0)$ is
contained in $A_l^p$ for $l$ large enough, by the lower
semicontinuity of the total variation on open sets (see for instance
\cite[Proposition~1.62(b)]{afp}) we find:
$$
|\Sigma|(B_r(0))\leq\liminf_{l\to\infty}|\Sigma^{n(l)}|(B_r(0))\leq
\liminf_{l\to\infty} |\Sigma^{n(l)}|(A_l^p)\leq |\Sigma|(\R^k)-2^{-p}/m.
$$
Letting $r\uparrow\infty$ we obtain a contradiction.
Therefore (\ref{sig2}) is proven.

In order to prove \eqref{cl}, taking again into
account the representation of $|\Sigma^U_h|(\R^k)$ given by the previous
proposition and the dominated convergence theorem, it suffices to show
that, with $y\in h^\perp$ fixed, $\min\limits_{t\in\R} V_n(y+th)$ converges
as $n\to\infty$ to $\min\limits_{t\in\R} V(y+th)$. By monotonicity we need
only to show that
\begin{equation}\label{gamali}
\liminf_n\min\limits_{t\in\R}V_n(y+th)\geq\min\limits_{t\in\R} V(y+th).
\end{equation}
Let $n(k)$ be a subsequence along which the liminf is achieved, let $t_k$ be
minimizers of $t\mapsto V_{n(k)}(y+th)$, and assume
(possibly extracting one more subsequence) that $t_k\to t$. The lower semicontinuity
of $V_{n(p)}$ gives $\lim_k V_{n(k)}(y+t_kh)\geq\liminf_k V_{n(p)}(y+t_kh)\geq V_{n(p)}(y+th)$.
Letting $p\to \infty$ we obtain \eqref{gamali}.
Finally, the tightness estimate allows to pass from convergence of $\Sigma^n$ in the
duality with $C_c(\R^k)$ to the convergence in the duality with $C_b(\R^k)$.
\end{proof}


\end{document}